\documentclass[11pt, reqno]{amsart}

\usepackage{amsmath, amsthm, amsfonts, amssymb, amsbsy, bigstrut, graphicx, enumerate,  upref, longtable, comment, booktabs, array, caption, subcaption}

\usepackage{pstricks}
\usepackage{times}
\allowdisplaybreaks

\usepackage[breaklinks]{hyperref}

\usepackage[numbers, sort&compress]{natbib} 

\newcommand{\mm}{\mathcal{M}}

\newtheorem{thm}{Theorem}[section]
\newtheorem{lmm}[thm]{Lemma}
\newtheorem{cor}[thm]{Corollary}
\newtheorem{prop}[thm]{Proposition}
\newtheorem{defn}[thm]{Definition}

\theoremstyle{definition}

\newcommand{\cc}{\mathbb{C}}

\newcommand{\mf}{\mathcal{F}}

\newcommand{\mx}{\mathcal{X}}

\newcommand{\my}{\mathcal{Y}}

\newcommand{\rr}{\mathbb{R}}
\newcommand{\smallavg}[1]{\langle #1 \rangle}

\newcommand{\tr}{\operatorname{Tr}}

\newcommand{\ve}{\varepsilon}

\newcommand{\zz}{\mathbb{Z}}

\newcommand{\fb}{\mathfrak{B}}

\newcommand{\dist}{\operatorname{dist}}
\newcommand{\area}{\operatorname{area}}
\newcommand{\mb}{\mathcal{B}}

\numberwithin{equation}{section}

\renewcommand{\Re}{\operatorname{Re}}

\newcommand{\ham}{\operatorname{H}}

\usepackage[utf8]{inputenc}
\usepackage{amsmath}
\usepackage{amsfonts}
\usepackage{bbm}
\usepackage{mathtools}
\usepackage{tikz}
\usetikzlibrary{decorations.pathreplacing}
\usepackage{amsthm}
\usepackage[english]{babel}
\usepackage{fancyhdr}
\usepackage{amssymb}
\usepackage{enumitem}
\usepackage{qtree}
\usepackage{mathrsfs}

\DeclareMathOperator{\supp}{supp}

\renewcommand{\tilde}{\widetilde}

\begin{document}

%\title[Mass gap implies quark confinement]{Mass gap and center symmetry imply quark confinement}
\title{A probabilistic mechanism for quark confinement}
\author{Sourav Chatterjee}
%\address{Stanford University}
%\date{\today}
\address{Departments of Mathematics and Statistics, Stanford University}
\email{souravc@stanford.edu}
\thanks{Research was partially supported by NSF grant DMS-1855484}

\begin{abstract}
The confinement of quarks is one of the enduring mysteries of modern physics. There is a longstanding physics heuristic that confinement is a consequence of `unbroken center symmetry'. This article gives mathematical  confirmation  of this heuristic, by rigorously defining of center symmetry in lattice gauge theories and proving that a theory is confining when center symmetry is unbroken. Furthermore, a sufficient condition for unbroken center symmetry is given: It is shown that if the center of the gauge group is nontrivial, and correlations decay exponentially under arbitrary boundary conditions, then center symmetry does not break. % --- and hence, quarks confine. %, it is shown that if correlations decay exponentially in a certain sense, then for local operators under arbitrary boundary conditions, and the gauge group is a compact connected matrix Lie group with a nontrivial center, then the theory is confining. The exponential decay assumption is stronger than usual mass gap, which means exponential decay under a specific boundary condition dictated by the Feynman path integral associated with the theory. The strengthening ensures that center symmetry is not spontaneously broken. This gives the first mathematical support for the longstanding belief in physics that mass gap plus unbroken center symmetry imply confinement. It also gives the first evidence  that correlation decay for local operators can prevent the breaking of center symmetry and cause confinement. The main step in the argument uses correlation decay for local operators to deduce  correlation decay for certain nonlocal operators similar to Polyakov loops. The proof is almost entirely based in probability theory, making extensive use of the idea of coupling probability measures. %This gives mathematical confirmation of a long-held belief in physics. %The proof makes extensive use of the idea of coupling from probability theory. %This gives mathematical justification for a longstanding  belief in physics about the mechanism behind confinement, which says, roughly, that confinement is the result of strong coupling behavior plus center symmetry. 
\end{abstract} 

\maketitle

%\tableofcontents

\section{Introduction}\label{intro}
Quantum gauge theories, also known as quantum Yang--Mills theories, are components of the Standard Model of quantum mechanics. In spite of many decades of research, physically relevant quantum gauge theories have not yet been constructed in a rigorous mathematical sense. The most popular approach to solving this problem is via the program of constructive field theory~\cite{gj87}. In this approach, one starts with a statistical mechanical model on the lattice; the next step is to pass to a continuum limit of this model; the third step is to show that the continuum limit satisfies certain `axioms'; if these axioms are satisfied, then there is a standard machinery which allows the construction of a quantum field theory. Taking this program to its completion is one of the Clay millennium problems~\cite{jaffewitten, chatterjee19b}. There are various well-known difficulties in adapting the constructive field theory approach to gauge theories, which is probably why the question has remained open for so long; for an alternative (but also unfinished) approach, see~\cite{seiler82}. 

The statistical mechanical models considered in the first step of the above program are known as {\it lattice gauge theories} (defined in Section \ref{resultsec}). Lattice gauge theories were introduced by \citet{wegner71} to study phase transitions without a local order parameter, and later reintroduced by \citet{wilson74} to study quark confinement. A lattice gauge theory may be coupled with, for example, a Higgs field, or it may be a {\it pure} lattice gauge theory. We will only deal with pure lattice gauge theories (also called lattice Yang--Mills theories) in this manuscript. A pure lattice gauge theory is characterized by its gauge group (usually a compact matrix Lie group), the dimension of spacetime, and a parameter known as the coupling strength.  These theories on their own, even without passing to the continuum limit or constructing the quantum theory, can yield substantial physically relevant information~\cite{greensite11}. Two of the most important open questions in this area have lattice gauge theoretic formulations. The first is the question of Yang--Mills mass gap. In lattice gauge theories, mass gap means exponential decay of correlations. Mass gap is not hard to establish at sufficiently large  values of the coupling strength (for example, using the methods of \cite{ds85}). However, it is widely believed~\cite{chatterjee19b, jaffewitten} (with some dissent~\cite{ps98}) that certain lattice gauge theories have mass gap at {\it all} values of the coupling strength. Perhaps the most important example is four-dimensional $SU(3)$ lattice gauge theory. If one can show that this theory has a mass gap at all values of the coupling strength, that would explain why particles known as {\it glueballs} in the theory of strong interactions have mass~\cite{mp99}. All such questions remain open. 

The second big open question is the problem of {\it quark confinement}. Quarks are the constituents of various elementary particles, such as protons and neutrons. It is an enduring mystery why quarks are never observed freely in nature. The problem of quark confinement has received enormous attention in the physics literature, but the current consensus seems to be that a satisfactory theoretical explanation does not exist~\cite{greensite11}. 

\citet{wilson74} argued that quark confinement is equivalent to showing that the relevant lattice gauge theory satisfies what's now known as {\it Wilson's area law} (defined in Section~\ref{resultsec}).  Soon after Wilson's work,   \citet{osterwalderseiler78} proved that the area law is always satisfied at sufficiently large coupling strength. However, to prove quark confinement, one needs to show that the area law holds at {\it all} values of the coupling strength --- and in particular, for very small values. (Actually, what is really needed is that the area law holds at coupling strengths arbitrarily close to a critical value; in many theories of interest, the critical value is believed to be zero.) This need not always be true; for example, \citet{guth80} and \citet{frohlichspencer82} showed that four-dimensional $U(1)$ lattice gauge theory is not confining at weak coupling. 

Proving that the area law holds at weak coupling remains a largely open problem above dimension two (where it is relatively easy --- see \cite[Section 6]{osterwalderseiler78}, and also~\cite{bdi74}). A rare instance where the area law has been shown to hold at weak coupling is three-dimensional $U(1)$ lattice gauge theory~\cite{gopfertmack82}. But the most important case of four-dimensional $SU(3)$ theory remains out of reach. 

Some of the other notable advances in the mathematical study of confinement include the work of \citet{frohlich79}, who showed that confinement holds in $SU(n)$ theory if it holds in the corresponding $\zz_n$ theory; the work of \citet{df80}, who showed that confinement in a $d$-dimensional pure lattice gauge theory holds if there is exponential decay of correlations in a $(d-1)$-dimensional nonlinear $\sigma$ model; the work of \citet{bs83}, who studied long range order for lattice gauge theories on cylinders; and the work of \citet{bf80} on the related problem of Debye screening. A toy model exhibiting a sharp transition from the confining to the deconfining regime was studied by \citet{accfr83}.  For some recent progress towards understanding confinement in large $n$ lattice gauge theories, see~\cite{chatterjee19a, cj16}.

%The existence of a mass gap is a feature of the strong coupling regime. Therefore it seems that mass gap must at least be a necessary condition for confinement. 

In physics, it has been believed since the work of \citet{thooft78} that a lattice gauge theory is confining when a certain kind of symmetry, known as {\it center symmetry}, is not spontaneously broken. This paper gives mathematical confirmation of this heuristic, by first giving a rigorous definition of center symmetry in lattice gauge theories, and then proving that a theory is confining if center symmetry is unbroken. Furthermore, a sufficient condition for unbroken center symmetry is provided: It is shown that if the center of the gauge group is nontrivial, and correlations decay exponentially under arbitrary boundary conditions, then center symmetry does not break, and therefore the theory is confining.

\section{Definitions and results}\label{resultsec}
This section contains the main results of this paper. We begin with the definitions of lattice gauge theory and Wilson loop variables. 
\subsection{Lattice gauge theories}\label{lgtdefsec}
Let $n\ge1$ and $d\ge 2$ be two integers. Let $G$ be a closed connected subgroup of $U(n)$. Let $E$ be the set of directed nearest-neighbor edges of $\zz^d$, where the direction is from the smaller vertex to the bigger one in the lexicographic ordering. We will call such edges {\it positively oriented}. For a positively oriented edge $e\in E$, let $e^{-1}$ denote the same edge but directed in the opposite direction. Such edges will be called {\it negatively oriented}.  Let $\Omega$ be the set of all functions from $E$ into $G$. That is, an element $\omega\in \Omega$ assigns a matrix $\omega_e\in G$ to each edge $e\in E$. If $\omega\in \Omega$ and $e$ is a negatively oriented edge, we define $\omega_e := \omega_{e^{-1}}^{-1}$. 

\begin{figure}[t!]
\begin{center}
\begin{tikzpicture}[scale = 3]
\draw[ultra thick, ->] (0,0) to (1,0);
\draw[ultra thick, ->] (1,0) to (1,1);
\draw[ultra thick, ->] (1,1) to (0,1);
\draw[ultra thick, ->] (0,1) to (0,0);
\node at (.5,-.15) {$e_1$};
\node at (1.15,.5) {$e_2$};
\node at (.5,1.15) {$e_3$};
\node at (-.15,.5) {$e_4$};
\node at (.5,.5) {$p$};
\end{tikzpicture}
\caption{A plaquette $p$ bounded by four directed edges $e_1,e_2,e_3,e_4$. \label{plaquettepic}}
\end{center}
\end{figure}
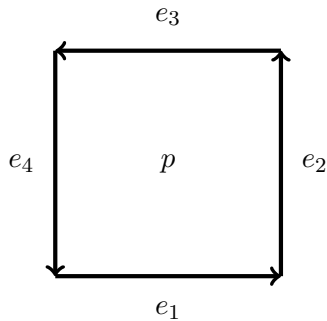

A {\it plaquette} in $\zz^d$ is a set of four directed edges that form the boundary of a square. Let $P$ be the set of all plaquettes. Given some $p\in P$ and $\omega \in \Omega$, we define $\omega_p$ as follows. Write $p$ as a sequence  of directed edges $e_1,e_2,e_3,e_4$, each one followed by the next  (see Figure \ref{plaquettepic}). Then let  $\omega_p := \omega_{e_1} \omega_{e_2}\omega_{e_3}\omega_{e_4}$. Although there are ambiguities in this definition about the choice of $e_1$ and the direction of traversal, that is not problematic because we will only use the quantity $\Re(\tr(\omega_p))$, which is not affected by these ambiguities. (Here $\omega_e$'s are $n\times n$ matrices, and $\tr$ denotes the trace of a matrix.)

Endow the product space $\Omega = G^E$ with the product $\sigma$-algebra and let $\lambda$ denote the normalized product Haar measure on $\Omega$. Pure lattice gauge theory on $\zz^d$ with gauge group $G$ and coupling parameter $\beta$ (equal to the inverse of the squared coupling strength) is formally defined as the probability measure $d\mu(\omega) = Z^{-1} e^{-\beta H(\omega)} d\lambda(\omega)$ on $\Omega$, where $H$ is the formal Hamiltonian
\begin{align}\label{hndef}
H(\omega) :=  -\sum_{p\in P} \Re(\tr(\omega_p))
\end{align}
and $Z$ is the normalizing constant. In the language of rigorous mathematical physics, what this actually defines is a {\it specification}: Although the definition of the probability measure $\mu$ as stated above does not make sense since the series defining $H$ may not be convergent, the conditional distribution of any finite set of $\omega_e$'s given all other $\omega_e$'s, under such a hypothetical probability measure $\mu$, is perfectly well-defined. Any actual probability measure $\mu$ on $\Omega$ which has these specified conditional distributions is called a {\it Gibbs measure} for this specification. It is not obvious that Gibbs measures exist. In the case of lattice gauge theories with compact gauge groups, the existence of at least one Gibbs measure follows from standard results, such as \cite[Theorem 4.22]{georgii11}. 

Throughout the rest of this paper, we will assume that $G$, $\beta$ and $d$ are fixed. In particular, whenever we refer to the lattice gauge theory defined above, we will think of $\beta$ as a fixed number and part of the definition.

%We will say that the loop is contained in $B_N$ if all its edges are interior edges of $B_N$. 
\subsection{Wilson loops and area law}\label{wilsonsec}
Let $\pi$ be a finite-dimensional irreducible unitary representation of the group $G$, and let $\chi_\pi$ be the character of $\pi$. Along with $G$, $\beta$ and $d$, the representation $\pi$ will remain fixed throughout this manuscript. %The simplest case is where $\pi$ is just the identity map on $G$, giving the $n$-dimensional representation where $G$ acts on $\cc^n$ in the usual way. 

%The pure lattice gauge theory defined above is coupled to a Higgs field (usually, a field of {\it quarks}) using a representation such as $\pi$. We will not deal with this coupled theory, because in the limit where the quark mass tends to infinity, the gauge field and the Higgs field decouple, and one can treat the gauge field as if it is a pure gauge theory. We will imagine that we are working in this limit. However, the representation $\pi$ is still important, because it is needed in Wilson's formulation of quark confinement. %The simplest scenario is when $\pi$ is just the usual action of $G$ on $\cc^n$. 

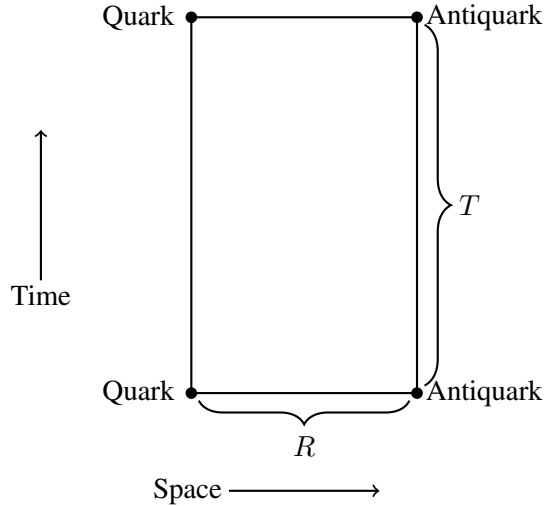
\begin{figure}[t!]
\begin{center}
\begin{tikzpicture}[scale = 1]
\draw[thick] (0,0) rectangle (3,5);
\draw[fill] (0,0) circle [radius = .07];
\draw[fill] (3,0) circle [radius = .07];
\draw[fill] (3,5) circle [radius = .07];
\draw[fill] (0,5) circle [radius = .07];
\node at (-.7,0) {Quark};
\node at (-.7,5) {Quark};
\node at (3.9,0) {Antiquark};
\node at (3.9,5) {Antiquark};
\draw [thick, decorate, decoration = {brace, mirror, amplitude = 9pt}] (0.1,-.1) to (2.9,-.1);
\node at (1.5, -.7) {$R$};
\draw [thick, decorate, decoration = {brace, mirror, amplitude = 10pt}] (3.1,.1) to (3.1,4.9);
\node at (3.7, 2.5) {$T$};
\draw[thick, ->] (-2,1.5) to (-2,3.5);
\draw[thick, ->] (.5, -1.3) to (2.5,-1.3);
\node at (-.05,-1.3) {Space};
\node at (-2,1.3) {Time};
\end{tikzpicture}
\caption{Rectangular Wilson loop representing a static quark-antiquark pair separated by distance $R$ over a time  interval of length $T$.\label{quarkpic}} %The expectation of the Wilson loop variable behaves like $e^{-V(R)T}$ where $V(R)$ is the potential energy of the static quark-antiquark pair. Thus, if the area law holds, the $V(R)$ is linear in $R$, which prevents the pair from separating beyond a certain distance.\label{quarkpic}}
\end{center}
\end{figure}

A {\it loop} $\ell$ in $\zz^d$ is a  sequence of directed edges $e_1,e_2,\ldots,e_k$ such that the end vertex of $e_i$ is the same as the beginning vertex of $e_{i+1}$ for $i=1,\ldots, k-1$, and the end vertex of $e_k$ is the beginning vertex of $e_1$.  Given a loop $\ell$ and a configuration $\omega\in \Omega(\zz^d)$, the {\it Wilson loop variable} $W_\ell(\omega)$ is defined as
\begin{align}
W_\ell(\omega) &:= \chi_\pi(\omega_{e_1}\omega_{e_2}\cdots \omega_{e_k})\notag \\
&= \tr(\pi(\omega_{e_1})\pi(\omega_{e_2}) \cdots \pi(\omega_{e_k})). \label{wilsondef}
\end{align}
Take any  Gibbs measure $\mu$ of our lattice gauge theory. The expected value $\smallavg{W_\ell}$ of a Wilson loop variable under $\mu$ is known as a {\it Wilson loop expectation}. A loop is called {\it rectangular} if it forms the boundary of a rectangle. The area enclosed by a rectangular loop is the product of the length and the breadth of the rectangle. The Gibbs measure is said to satisfy {\it Wilson's area law} for the representation $\pi$ if 
\begin{align}\label{wilson0}
|\smallavg{W_\ell}| \le C_1 e^{-C_2\area(\ell)}
\end{align}
for any rectangular loop $\ell$, where $C_1$ and $C_2$ are positive constants that do not depend on $\ell$, and $\area(\ell)$ is the area enclosed by $\ell$.

The reason why the area law is thought to imply confinement of quarks is as follows. Let $V(R)$ be the potential energy of a static quark-antiquark pair separated by distance $R$. Then quantum field theoretic calculations~\cite{wilson74}  indicate that for a rectangular loop $\ell$ with side-lengths $R$ and $T$, $\smallavg{W_\ell}$ should behave like $e^{-V(R)T}$ when $R$ is fixed and $T$ is taken  to infinity. Here the sides of length $R$ represent the lines joining the quark-antiquark pair at times $0$ and $T$, and the sides of length $T$ represent the trajectories of the quark and the antiquark in the time direction (Figure \ref{quarkpic}). So if the area law holds, then  $V(R)$ grows linearly in the distance $R$ between the quark and the antiquark. By the conservation of energy, this implies that the pair will not be able to separate beyond a certain distance. 

Notice that for quarks to be confined, it is not really necessary to have $V(R)$ growing linearly with $R$. It suffices to have $V(R)\to \infty$ as $R\to\infty$. Thus, for a given Gibbs measure to be confining, it is sufficient to have that for any $R$,
\[
\limsup_{T\to\infty} \frac{1}{T}\log|\smallavg{W_{\ell_{R,T}}}|\le -V(R)
\]
for some $V$ with $V(R) \to \infty$ as $R\to\infty$, where $\ell_{R,T}$ is any rectangular loop with side lengths $R$ and $T$.
%Renormalization group arguments predict that $\beta$ has to be sent to infinity as the lattice spacing goes to zero to obtain the continuum limit of four-dimensional non-Abelian theories~\cite{balaban88}; that is why we need the area law to hold at arbitrarily large values of $\beta$. 

\subsection{Unbroken center symmetry implies confinement}\label{cssec}
Physics literature tells us that confinement happens when a certain symmetry of a lattice gauge theory, known as {\it center symmetry}, is not broken. To my knowledge, there is no rigorous definition of center symmetry in the mathematical physics literature, nor a discussion of what it means for this symmetry to be broken. I will now propose a  definition, inspired by the physical definition of center symmetry via Polyakov loops~\cite{greensite11}. 

Instead of the lattice $\zz^d$, fix a positive integer $N$ and consider the {\it slab} $S := \{0,1,\ldots,N\} \times \zz^{d-1}$. Let $E(S)$ denote the set of positively oriented edges of $S$ and $P(S)$ denote the set of plaquettes in $S$.

The slab $S$ has two boundaries, namely, the top boundary $\{N\}\times \zz^{d-1}$, and the bottom boundary $\{0\}\times \zz^{d-1}$. Let $\partial E(S)$  denote the set of all boundary edges. Let us call an element of $G^{\partial E(S)}$ a {\it boundary condition}. Given a boundary condition $\delta$, let $\Omega(S,\delta)$ denote the set of all $\omega\in G^{E(S)}$ that agree with $\delta$ on the boundary. The formal Hamiltonian
\[
H_{S,\delta}(\omega) := -\sum_{p\in P(S)} \Re(\tr(\omega_p))
\]
defines a specification on $\Omega(S,\delta)$ with coupling parameter $\beta$ in the usual way. Let us refer to this specification as our lattice gauge theory on $S$ with boundary condition $\delta$, and denote it by $\alpha$. 

Now consider the following transformation on $\Omega(S,\delta)$. Take any element $g_0$ in the center of the group $G$. Take any $\omega\in \Omega(S,\delta)$. For each edge $e$ from the boundary $\{0\}\times \zz^{d-1}$ to the layer $\{1\}\times \zz^{d-1}$, replace $\omega_e$ by $g_0\omega_e$. Call the resulting configuration $\tau(\omega)$. We will refer to $\tau$ as a {\it center transform}. 

Since $g_0$ is an element of the center of $G$, it is easy to see that the map $\tau$ leaves $\omega_p$ invariant for any plaquette $p$, and hence (formally) leaves the Hamiltonian $H_{S,\delta}$ invariant. From this, it follows that $\tau$ is a {\it symmetry} of the specification $\alpha$ defined above (see \cite[Chapter 5]{georgii11} for the definition of a symmetry of a specification). We will call this a {\it center symmetry} of our lattice gauge theory in the slab $S$. Following the definition of unbroken symmetry in mathematical physics~\cite[Definition 5.21]{georgii11}, we will say that the specification $\alpha$ has unbroken center symmetry if every Gibbs measure of $\alpha$ is invariant under the action of each center transform. 

%Finally, we are ready to state the main definition of this subsection.
\begin{defn}\label{csdef}
We will say that the lattice gauge theory on $\zz^d$ defined in Subsection \ref{lgtdefsec} has unbroken center symmetry if for some $N\ge 1$, the theory restricted to the slab $S = \{0,1,\ldots, N\}\times \zz^{d-1}$, under any boundary condition, has unbroken center symmetry as defined above.
\end{defn}
Incidentally, the restriction of Gibbs measures to slabs has similarities with `finite temperature' states in lattice gauge theories, where the possible breaking of the analogous center symmetry has been analyzed in~\cite{bs83}.

Having defined center symmetry and what it means for it to be unbroken, we now arrive at the first main result of this paper. 
\begin{thm}\label{main1}%[Unbroken center symmetry implies confinement]
Suppose that the lattice gauge theory defined in Subsection \ref{lgtdefsec} has unbroken center symmetry in the sense of Definition \ref{csdef}. Let $\pi$  be a finite-dimensional irreducible unitary representation of $G$ that acts nontrivially on the center of $G$. Let $W_\ell$ denote the Wilson loop variable for a loop $\ell$, defined using the representation $\pi$ as in equation \eqref{wilsondef}. Then there is a function $V:\zz_+\to \rr$, satisfying $V(R)\to \infty$ as $R\to \infty$, such that for any rectangular loop $\ell$ with side-lengths $R\le T$, and for any Gibbs measure of our lattice gauge theory on $\zz^d$, we have $|\smallavg{W_\ell}|\le e^{-V(R) T}$. 
\end{thm}
In the above, the condition `$\pi$ acts nontrivially on the center of $G$' means that there is at least one element $g_0$ in the center such that $\pi(g_0)$ is not the identity operator. This requires, in particular, that the center of $G$ is nontrivial. Indeed, it is believed that if the center of the gauge group is trivial, such as in $SO(3)$ theory, then quarks may not confine at weak coupling~\cite{gl81}.

\subsection{A sufficient condition for unbroken center symmetry}\label{suffsec}
We will say that two edges of $\zz^d$ are neighbors if they both belong to some common plaquette. A measurable map $f:\Omega \to \rr$ will be called a {\it local function} supported on an edge $e\in E$ if $f(\omega)$ depends only on the values of $\omega_u$ for $u$ that are neighbors of $e$. Given two local functions $f$ and $g$, let $\dist(f,g)$ denote the Euclidean distance between the midpoints of their supporting edges. 

We will say that a subset of $\zz^d$ is a cube if it is a translate of $\{0,\ldots, N\}^d$ for some $N$. A boundary condition on a cube $B$ is an assignment of elements of $G$ to the boundary edges of $B$.  Given a cube $B$ and a boundary condition $\delta$, let $\Omega_{B,\delta}$ denote the set of all assignments of elements of $G$ to edges of $B$ which agree with $\delta$ on the boundary. The lattice gauge theory of Subsection~\ref{lgtdefsec} defines a probability measure on $\Omega_{B,\delta}$ in the natural way. For a measurable map $f:\Omega_{B,\delta} \to \rr$, let $\smallavg{f}_{B,\delta}$ denote the expectation of $f$ under this measure (provided that the expectation exists). 
%A function $f:\Omega \to \rr$ is called a local function if the value of $f(\omega)$ depends only on the values of finitely many $\omega_e$'s. Call this set of edges the `support of $f$'. If $f$ and $g$ are local functions, let $\dist(f,g)$ denote the distance between their supports, viewing the supports as closed subsets of $\rr^d$. Usually, we say that a Gibbs measure has exponentially decaying two-point correlations of $\smallavg{fg}-\smallavg{f}\smallavg{g}$ decays exponentially in $\dist(f,g)$. Let us assume something a bit stronger.
\begin{defn}\label{expdecaydef}
Consider the lattice gauge theory defined in Subsection \ref{lgtdefsec}. We will say that this theory satisfies exponential decay of correlations under arbitrary boundary conditions if there are positive constants $K_1$ and $K_2$ depending only on $G$, $\beta$ and $d$, such that  for any cube $B$, for any boundary condition $\delta$ on $B$, and for any local functions $f$ and $g$ supported on edges in $B$ and taking values in $[-1,1]$, we have  $|\smallavg{fg}_{B,\delta} - \smallavg{f}_{B,\delta}\smallavg{g}_{B,\delta}|\le K_1 e^{-K_2\dist(f,g)}$.
\end{defn}
The above condition is quite strong, but there are situations where it can be proven to hold. For example, it can be shown to hold for essentially any lattice gauge theory when $\beta$ is small enough, using the methods of  \citet{ds85}.  There is also a belief that some form of exponential decay of correlations should hold at large $\beta$ in four-dimensional non-Abelian theories~\cite{jaffewitten, chatterjee19b} --- this is the Yang--Mills mass gap conjecture --- but it is not clear whether this belief includes the strong version stated above.

We now arrive at the second main result of this paper, which says that exponential decay of correlations under arbitrary boundary conditions implies unbroken center symmetry.
\begin{thm}\label{main2}
Consider the lattice gauge theory defined in Subsection \ref{lgtdefsec}. Suppose that it satisfies exponential decay of correlations under arbitrary boundary conditions, according to Definition \ref{expdecaydef}. Then it has unbroken center symmetry. Moreover, if Wilson loop variables are defined using a finite-dimensional irreducible unitary representation $\pi$ that acts nontrivially on the center of $G$, then any Gibbs measure for the theory satisfies Wilson's area law \eqref{wilson0} for rectangular loops. 
\end{thm}
Incidentally, there are lattice gauge theories that are believed to be `gapped' --- that is, exhibiting exponential decay of correlations of local observables --- and possessing center symmetry, but not confining at large $\beta$. Examples include lattice gauge theories with finite gauge groups~\cite{seiler82, flv20, cao20, chatterjee20}, and theories coupled with Higgs fields in the adjoint representation~\cite{fradkinshenker79}. If this is true, then by Theorem~\ref{main1}, it must be that center symmetry breaks spontaneously in these models. The reason why this does not contradict Theorem~\ref{main2} is that the `exponential decay of correlations' in these theories refer to exponential decay of two-point truncated correlations under certain kinds of boundary conditions~\cite{chatterjee20, cao20}; this is different than what's commonly understood as `decay of correlations' in mathematical physics, which means absence of long-range order, or equivalently, the uniqueness of the Gibbs measure~\cite{georgii11}. Indeed, some of the above theories have been rigorously shown to be possessing multiple Gibbs measures~\cite{borgs84, ks82} and therefore actually have long-range order at weak coupling.

This completes the statements of results. The rest of the paper is devoted to the proofs of Theorem \ref{main1} and Theorem \ref{main2}.

\section{Proof of Theorem \ref{main1}}\label{main1proof}
Let $S = \{0,\ldots,N\}\times \zz^{d-1}$ be a slab where our lattice gauge theory has unbroken center symmetry under any boundary condition. Let $m$ be the dimension of the representation $\pi$. Let $g_0$ be an element in the center of $G$ such that $\pi(g_0)$ is not the identity operator. Since $\pi(g_0)$ commutes with $\pi(g)$ for every $g\in G$, Schur's lemma~\cite[Corollary 4.30]{hall15} implies that $\pi(g_0) = cI$ for some $c\ne 1$, where $I$ is the $m\times m$ identity matrix. %Since $\pi$ is a unitary representation, we must have $|c|=1$. 

We will say that an edge $e$ is a {\it vertical edge} if its endpoints differ in the first coordinate, imagining the first coordinate as the vertical direction. A sequence of vertical edges that connects the two boundaries of $S$ will be called a {\it vertical chain} of edges. Let $e_1,\ldots,e_{N}$ be a vertical chain. A {\it vertical chain variable} $f$ associated with this chain is a product of one element from each of the matrices $\pi(\omega_{e_1}),\ldots, \pi(\omega_{e_{N}})$. Since $\pi$ is an $m$-dimensional representation, there are $m^{2N}$ vertical chain variables associated with a given vertical chain. %Note that $\dist(e,\partial'E_{M,N})$ is the same for any edge $e$ in a given vertical chain $\xi$. We will denote this number by $\dist(\xi, \partial'E_{M,N})$.  (See Figure \ref{chainpic}.) 

For each $R\ge 1$, let $S_R:= \{0,\ldots, N\}\times \{-R,\ldots, R\}^{d-1}$. A boundary condition on $S_R$ is an assignment of elements of $G$ to the boundary edges of $S_R$.  Given a boundary condition $\delta$, let $\Omega_{R,\delta}$ denote set  of all assignments of elements of $G$ to edges of $S_R$ which agree with $\delta$ on the boundary. The lattice gauge theory of Subsection~\ref{lgtdefsec} defines a probability measure on $\Omega_{R,\delta}$ in the natural way. For a measurable map $f:\Omega_{R,\delta} \to \rr$, let $\smallavg{f}_{R,\delta}$ denote the expectation of $f$ under this measure (provided that the expectation exists).
\begin{lmm}\label{vertchainlmm}
Let $\xi$ be the unique vertical chain containing the origin. There is a function $V:\zz_+\to \rr$ satisfying $V(R)\to \infty$ as $R\to\infty$, such that for any vertical chain variable $f$ associated with $\xi$, any  $R$, and any boundary condition $\delta$ on $S_R$, we have $|\smallavg{f}_{R,\delta}|\le e^{-V(R)}$. 
\end{lmm}
\begin{proof}
A different way to state the claim is that 
\[
\lim_{R\to \infty} \sup_{f,\delta} |\smallavg{f}_{R,\delta}|=0,
\]
where the supremum is taken over all boundary conditions $\delta$ on $S_R$ and all vertical chain variables $f$ associated with $\xi$. Suppose that the claim is not true. Then the above version shows that there is a real number $\ve >0$, a sequence of integers $R_k\to\infty$,  a sequence of vertical chain variables $\{f_k\}_{k\ge 1}$ associated with $\xi$, and a boundary condition $\delta_k$ on  $S_{R_k}$ for each $k$, such that $|\smallavg{f_k}_{R_k,\delta_k}|\ge \ve$ for all $k$. Let $\delta_{k,e}$ denote the group element assigned by $\delta_k$ to an edge $e$ on the boundary of $S_{R_k}$. Passing to a subsequence if necessary, we may assume the following: 
\begin{itemize}
\item For each edge $e$ in the boundary of the infinite slab $S$, $\delta_{k,e}$ tends to a limit $\delta_e$ as $k\to\infty$. This defines a boundary condition $\delta$ on $S$. %(Note that each such $e$ is in the boundary of $S_{R_k}$ for all large enough $k$, so this makes sense.)
\item For each $k$, $f_k$ is the same variable $f$. We can assume this since there are only a finite number of vertical chain variables associated with $\xi$.
\end{itemize}
Now take any $R\ge 1$, and any $k$ so large that $R_k\ge R$. Since $|\smallavg{f}_{R_k,\delta_k}|\ge \ve$, and $\smallavg{f}_{R_k,\delta_k}$ is the weighted average of $\smallavg{f}_{R,\eta}$ over all boundary conditions $\eta$ on $S_R$ that agree with $\delta_k$ on the top and bottom faces, it follows that there is a boundary condition $\eta_k$ for which $|\smallavg{f}_{R,\eta_k}|\ge \ve$. Passing to a subsequence if necessary, we may assume that $\eta_k$ converges to a limit $\eta$ as $k\to\infty$. Since $\eta_k$ agrees with $\delta_k$ on the top and bottom faces of $S_R$, and $\delta_{k,e}\to \delta_k$ for each edge $e$ on the boundary of $S$, it follows that $\eta$ must agree with $\delta$ on the top and bottom faces of $S_R$. 

Thus, we have shown that for any $R$, there is some boundary condition $\eta_R$ on $S_R$ that agrees with $\delta$ on the top and bottom faces, such that $|\smallavg{f}_{R,\eta_R}|\ge \ve$. 

Recall the set $\Omega(S,\delta)$ consisting of all assignments of group elements to edges in $S$ that agree with $\delta$ on the boundary of $S$. Since $G$ is compact, $\Omega(S,\delta)$ is a compact metric space under the product topology. For each $R$, define a  $\Omega(S,\delta)$-valued random configuration $\omega$  as follows:
\begin{itemize}
\item For each edge $e$ outside $S_R$, let $\omega_e$ be the identity element of $G$.
\item On the boundary of $S_R$, let $\omega$ be equal to $\eta_R$.
\item Inside $S_R$, generate $\omega$ from the lattice gauge theory on $S_R$ with boundary condition $\eta_R$.
\end{itemize}
By the compactness of $\Omega(S,\delta)$ and Prokhorov's theorem~\cite{billingsley99}, this sequence of random configurations converges in law through a subsequence. The limit law is easily verified to be a Gibbs measure for our lattice gauge theory on $S$ with boundary condition $\delta$. Since $f$ is a bounded continuous function on $\Omega(S,\delta)$, it follows that the expectation of $f$ under this Gibbs measure has absolute value at least $\ve$. 

However, under the center transform described in Subsection \ref{cssec} using our chosen element $g_0$, $f$ transforms to $cf$. So if the center symmetry assumption holds, then we must have $\smallavg{f} = c\smallavg{f}$. Since $c\ne 1$, this implies that $\smallavg{f}=0$. This contradicts the conclusion of the previous paragraph.
\end{proof}

%We will now bootstrap the bound obtained in Lemma \ref{vertlmm} to show that the expected value of a vertical chain variable in a slab is exponentially decaying in both the distance of the chain to the spatial boundary, as well as the thickness of the slab. The Markovian nature of lattice gauge theory is the key ingredient in this proof. This dual decay is the chief reason behind the area law decay of Wilson loop expectations, as we will see in the next section. 

We will now bootstrap the bound from Lemma \ref{vertchainlmm} to get a bound on expected values of vertical chain variables in finite slabs of arbitrary thickness. 

\begin{lmm}\label{bootlmmnew}
Take any $T\ge N$ and $R\ge 1$. Let $\delta$ be any boundary condition on the slab $\{0,\ldots,T\}\times \{-R,\ldots,R\}^{d-1}$, and consider the lattice gauge theory on this slab with this boundary condition.  Then for any vertical chain variable $f$ defined on the vertical chain $\xi$ containing the origin, we have $|\smallavg{f}| \le e^{-V(R)[T/N]}$, 
where $V$ is the function from Lemma \ref{vertchainlmm} and $[T/N]$ is the integer part of $T/N$.
\end{lmm}
\begin{proof}
Write $T= qN + k$, where $q = [T/N]$ and $k<N$ are the quotient and the remainder when $T$ is divided  by $N$. Divide the slab into $q+1$ slabs $S_1,S_2,\ldots,S_{q+1}$, where $S_i$ has thickness $N$ for $i=1,\ldots,q$, and $S_{q+1}$ has thickness $k$. Let the top face of $S_i$, which is the same as the bottom face of $S_{i+1}$, be called $F_i$. Conditioning on $\omega_e$ for all $e\in \cup_{i=1}^{q+1}F_i$, the lattice gauge theory on the slab splits up into independent lattice gauge theories in the slabs $S_1,\ldots,S_{q+1}$. The vertical chain observable $f$ also splits up as a product $f_1f_2\cdots f_{q+1}$, where $f_i$ is the part of $f$ coming from $S_i$. Under the above conditioning, the random variables $f_1,\ldots,f_{q+1}$ are independent. (See Figure \ref{bootpicnew}.)

\begin{figure}[t!]
\begin{center}
\begin{tikzpicture}[scale = 1]
\draw[thick] (0,0) to (10,0);
\draw[thick] (0,1) to (10,1);
\draw[thick] (0,2) to (10,2);
\draw[thick] (0,3) to (10,3);
\draw[thick] (0,4) to (10,4);
\draw[thick] (0,0) to (0,4);
\draw[thick] (10,0) to (10,4);
\draw[dashed] (5,0) to (5,4);
\node at (5.25,.5) {$f_1$};
\node at (5.25,1.5) {$f_2$};
\node at (5.25,2.5) {$f_3$};
\node at (5.25,3.5) {$f_4$};
\draw [thick, decorate, decoration = {brace, amplitude = 4pt}] (-.1,0) to (-.1,1);
\draw [thick, decorate, decoration = {brace, mirror, amplitude = 7pt}] (10.1,0) to (10.1,4);
\draw [thick, decorate, decoration = {brace, mirror, amplitude = 9pt}] (0,-.1) to (10,-.1);
\node at (5, -.7) {$R$};
\node at (10.7, 2) {$T$};
\node at (-.65, .5) {$N$};
\draw[fill] (5,0) circle [radius = .07];
\draw[fill] (5,1) circle [radius = .07];
\draw[fill] (5,2) circle [radius = .07];
\draw[fill] (5,3) circle [radius = .07];
\draw[fill] (5,4) circle [radius = .07];
\end{tikzpicture}
\caption{Bootstrapping by subdividing into smaller slabs. \label{bootpicnew}}
\end{center}
\end{figure}
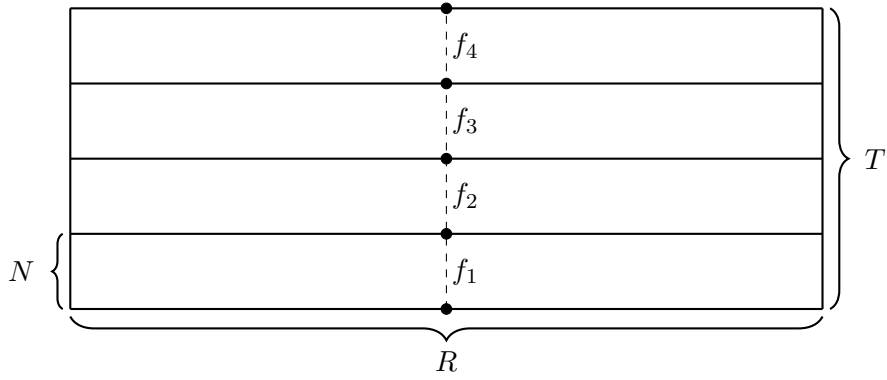

Now let $\smallavg{f_i}'$ denote the conditional expectation of $f_i$ under the above conditioning. For each $1\le i\le q$, Lemma \ref{vertchainlmm} gives $|\smallavg{f_i}'| \le e^{-V(R)}$. Also, since $\pi$ is a unitary representation, it follows that $|\smallavg{f_{q+1}}'|\le 1$. Thus,
\begin{align*}
|\smallavg{f}|&= |\smallavg{f_1f_2\cdots f_{q+1}}|\\
&= |\smallavg{\smallavg{f_1f_2\cdots f_{q+1}}'}|\\
&= |\smallavg{\smallavg{f_1}'\smallavg{f_2}'\cdots \smallavg{f_{q+1}}'}|\\
&\le \smallavg{|\smallavg{f_1}'||\smallavg{f_2}'|\cdots |\smallavg{f_{q+1}}'|}\le  e^{-V(R)q}.
\end{align*}
This completes the proof of the lemma. 
\end{proof}

%The next lemma is `almost' a proof of Theorem \ref{main1}, except for an additional term that we will remove afterwards. 
We are now ready to complete the proof of Theorem \ref{main1}.
%\begin{lmm}\label{almostlmm}
%Under the assumptions of Theorem \ref{main1}, we have that for any Gibbs measure, any $R$ and $T$, and any rectangular loop $\ell$ with side-lengths $R$ and $T$, 
%\[
%|\smallavg{W_\ell}|\le m^{2(R+T)}e^{ - V(R)[T/N]},
%\]
%where $V$ is the function from Lemma \ref{vertchainlmm} and $m$ is the dimension of $\pi$. 
%\end{lmm}
\begin{proof}[Proof of Theorem \ref{main1}]
Let $e_1,\ldots,e_k$ be the edges of $\ell$, so that $k=2(R+T)$. Let $f(\omega)$ be a product of one element from each of the matrices $\pi(\omega_{e_1}),\ldots,\pi(\omega_{e_k})$. Let us call any such $f$ a {\it component variable} associated with the loop $\ell$. Note that the Wilson loop variable $W_\ell$ is a sum of $m^{k}$ component variables. 

Let $S$ be a translate of the slab $\{0,\ldots,T\}\times \{-R,\ldots, R\}^{d-1}$ such that one of the vertical sides of $\ell$ passes through the center of $S$ and the other three sides belong to the boundary of $S$. Take any component variable $f$ associated with $\ell$. Then $f$ can be written as the product $f_1f_2$, where $f_1$ is the product of terms from the vertical side that passes through the center of $S$, and $f_2$ is the product of terms from the other three sides. Let $\smallavg{\cdot}'$ denote conditional expectation given $\omega_e$ for all $e$ that are either on the boundary of $S$ or outside $S$. Then
\begin{align}\label{f1f2idnew}
\smallavg{f} = \smallavg{f_1f_2} = \smallavg{\smallavg{f_1f_2}'} = \smallavg{\smallavg{f_1}' f_2},
\end{align}
where the second identity holds by the tower property of conditional expectation and the third identity holds because $f_2$ depends only on $\omega_e$ for edges $e$ that are on the boundary of $S$. Now note that $f_1$ is a vertical chain variable associated with a vertical chain passing through the center of $S$. Thus, by Lemma \ref{bootlmmnew}, 
\begin{align}\label{f1vr}
|\smallavg{f_1}'|&\le e^{ - V(R)[T/N]}.
\end{align}
Since $\pi$ is a unitary representation, we have $|f_2|\le 1$. Thus, by \eqref{f1f2idnew} and \eqref{f1vr}, we get the bound $|\smallavg{f}|\le e^{-V(R)[T/N]}$.  Since $W_\ell$ is a sum of $m^{2(R+T)}$ component variables, this gives
\[
|\smallavg{W_\ell}|\le m^{2(R+T)}e^{ - V(R)[T/N]}.
\]
Now recall that $R\le T$, and $[x]\ge x/2$ for any $x\ge 1$. Thus, 
\[
|\smallavg{W_\ell}|\le m^{4T}e^{ - V(R)T/2N}.
\]
In other words, $|\smallavg{W_\ell}| \le e^{-V_1(R)T}$, where $V_1(R) = -4\log m +V(R)/2N$. Since $V_1(R)\to \infty$ as $R\to \infty$, this completes the proof.
\end{proof}

\section{Paths in the gauge group}
%In this section we investigate the variation in a function on the gauge group $G$ depending on the size of its gradient. 
Let us denote the Euclidean norm of a vector $x\in \cc^n$ by $\|x\|$, and the Euclidean inner product of two vectors $x$ and $y$ by $x\cdot y$. Let $M_n(\cc)$ denote the space of all $n\times n$ complex matrices. We will identify $M_n(\cc)$ with $\rr^{2n^2}$ as a real manifold and view $G$ as a subset of $M_n(\cc)$. If $f:M_n(\cc)\to \rr$ is a smooth function, $\nabla f$ will denote the gradient of $f$, viewing $f$ as a function from $\rr^{2n^2}$ into~$\rr$. Lastly, for a matrix $A\in M_n(\cc)$, let $\|A\|$ denote the Euclidean (or Hilbert--Schmidt) norm of $A$ --- that is, the square-root of the sum of squared absolute values of the entries of~$A$. %Here we view $G$ as embedded in $M_n(\cc)$, and the function as a function from $M_n(\cc)$ into $\rr$. The gradient of $f$ refers to the gradient on the real manifold $M_n(\cc)$, as explained in Section \ref{notationsec}. The first step is proving the existence of well-behaved paths connecting any two elements of $G$. The key ingredient in the proof of this lemma is the closed subgroup theorem of Lie theory.
\begin{lmm}\label{pathlmm}
There is a finite constant $K$, depending only on $G$, such that the following is true. For any $g, h\in G$, there is a continuous path $\phi:[0,1]\to G$ such that $\phi(0)=g$, $\phi(1)=h$, $\phi$ is piecewise smooth with  a finite number of pieces, and $\|\phi'(t)\|\le K$ everywhere inside the smooth pieces.
\end{lmm}
\begin{proof}
Fix some $\ve>0$, to be chosen later. Let 
\[
U_\ve := \{g\in G: \|g-I\|< \ve\}.
\]
For any $g\in G$, let $U_\ve(g) := \{gh: h\in U_\ve\}$. Each $U_\ve(g)$ is a relatively open subset of $G$, and these open sets cover $G$. So by the compactness of $G$, there exist finitely many $g_1,\ldots, g_k\in G$ such that 
\[
G = \bigcup_{i=1}^k U_\ve(g_i).
\]
Let us put a graph structure on $\{U_\ve(g_i)\}_{1\le i\le k}$ by putting an edge between $U_\ve(g_i)$ and $U_\ve(g_j)$ when $U_\ve(g_i) \cap U_\ve(g_j)\ne \emptyset$. If this graph has more than one connected component, then $G$ can be written as a union of two disjoint nonempty relatively open sets, which is impossible since $G$ is connected. Thus, the graph constructed above must be connected.

Take any $g,h\in G$. Find $i$ and $j$ such that $g\in U_\ve(g_i)$ and $h\in U_\ve(g_j)$. By the above paragraph, there is a sequence $i=i_0, i_1,\ldots,i_m = j$ such that $U_\ve(g_{i_a})\cap U_\ve(g_{i_{a+1}})\ne \emptyset$ for each $a$.  Choose some $h_a\in U_\ve(g_{i_a})\cap U_\ve(g_{i_{a+1}})$ for $a = 0,1,\ldots,m-1$. (See Figure \ref{liepic}.)

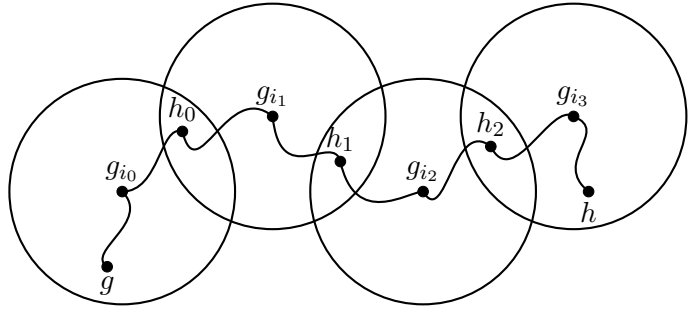
\begin{figure}[t!]
\begin{center}
\begin{tikzpicture}[scale = 1]
\draw [fill] (0,0) circle [radius = .07] node [black, above] {$g_{i_0}$};
\draw [thick] (0,0) circle [radius = 1.5];
\draw [fill] (.8,.8) circle [radius = .07] node [black, above] {$h_0$};
\draw [thick] (2,1) circle [radius = 1.5];
\draw [fill] (2,1) circle [radius = .07] node [black, above] {$g_{i_1}$};
\draw [fill] (2.9,.4) circle [radius = .07] node [black, above] {$h_1$};
\draw [fill] (4,0) circle [radius = .07] node [black, above] {$g_{i_2}$};
\draw [thick] (4,0) circle [radius = 1.5];
\draw [fill] (6,1) circle [radius = .07] node [black, above] {$g_{i_3}$};
\draw [thick] (6,1) circle [radius = 1.5];
\draw [fill] (4.9,.6) circle [radius = .07] node [black, above] {$h_2$};
\draw [fill] (-.2,-1) circle [radius = .07] node [black, below] {$g$};
\draw [fill] (6.2,0) circle [radius = .07] node [black, below] {$h$};
\draw [thick] (-.2,-1) .. controls (-.5,-.8) and (.4,-.4) .. (0,0);
\draw [thick] (0,0) .. controls (.5,0) and (.5,.9) .. (.8,.8);
\draw [thick] (.8,.8) .. controls (1,0) and (1.4,1.5) .. (2,1);
\draw [thick] (2,1) .. controls (2.1,0) and (2.8,.8) .. (2.9,.4);
\draw [thick] (2.9,.4) .. controls (3.1,-.5) and (3.8,0) .. (4,0);
\draw [thick] (4,0) .. controls (4.3,-.5) and (4.4,1) .. (4.9,.6);
\draw [thick] (4.9,.6) .. controls (5.2,0) and (5.7,1.2) .. (6,1);
\draw [thick] (6,1) .. controls (6.5,.8) and (5.7,.4) .. (6.2,0);
\end{tikzpicture}
\caption{A piecewise smooth path from $g$ to $h$, with $m=3$.\label{liepic}}
\end{center}
\end{figure}

By the closed subgroup theorem of Lie theory, there exist $\ve, \ve' >0$ such that for any $g\in U_\ve$, there is some $X\in M_n(\cc)$ with $\|X\|< \ve'$, such that $e^X = g$ and $e^{tX}\in G$ for all $t\in \rr$.  (For example, see \cite[Corollary 3.44]{hall15}.) Let this $\ve$ be our chosen $\ve$.

Take any $0\le a\le m-1$. Since $h_a \in U_\ve(g_{i_a})$, there is some $u_a \in U_\ve$ such that $h_a = g_{i_a}u_a$. By the previous paragraph, $u_a = e^{X_a}$ for some  $X_a\in M_n(\cc)$ with $\|X_a\|<\ve'$,  such that $e^{tX_a}\in G$ for all $t\in \rr$. Define a path $\phi_a$ from $g_{i_a}$ to $h_a$ as $\phi_a(t) := g_{i_a}e^{tX_a}$, for $0\le t\le 1$. This is a smooth path in $G$ connecting $g_{i_a}$ to $h_a$, and for each $t$, 
\[
\|\phi_a'(t)\| = \|g_{i_a}X_ae^{tX_a}\|=\|X_a\|\le \ve',
\]
where the second equality holds because $g_{i_a}$ and $e^{tX_a}$ are unitary matrices.

We can similarly construct paths connecting $h_a$ to $g_{i_{a+1}}$ for each $a$, and paths connecting $g$ to $g_i$ and $g_j$ to $h$. These paths will all satisfy the above bound on the norm of the derivative. Therefore if we join these $2m+2$ paths, we get a piecewise smooth path $\phi:[0,2m+2] \to G$ connecting $g$ to $h$, such that $\|\phi'(t)\|\le \ve'$ for all $t$ inside the smooth pieces. Rescaling the parameter, we get a piecewise smooth $\phi:[0,1]\to G$ connecting $g$ to $h$ such that $\|\phi'(t)\|\le (2m+2)\ve'$ for all $t$ inside the smooth pieces. Since $m\le k$, and $k$ and $\ve'$ depend only on $G$, this completes the proof. 
\end{proof}
As a consequence of the above lemma, we obtain the following corollary, which says that if $f:M_n(\cc) \to \rr$ is a smooth function, then the diameter of the set $f(G)$ is bounded by a constant multiple of the maximum value of $\|\nabla f\|$ on $G$. %This technical result will be enormously useful for us in later sections.
\begin{cor}\label{liecor}
Let $K$ be as in Lemma \ref{pathlmm}. Let $f:M_n(\cc)\to \rr$ be a smooth function. Then for any $g,h\in G$, $|f(g)-f(h)|\le K \|\nabla f\|_G$, where $\|\nabla f\|_G := \sup\{\|\nabla f(g)\|: g\in G\}$.
\end{cor}
\begin{proof}
Take any $g,h\in G$. Let $\phi$ be as in Lemma \ref{pathlmm}. There are numbers $0=t_0<t_1<t_2<\cdots < t_k =1$ for some $k$, such that $\phi$ is smooth in each interval $(t_i, t_{i+1})$. By the continuity of $\phi$,
\begin{align*}
f(h)- f(g) &= \sum_{i=0}^{k-1} (f(\phi(t_{i+1})) - f(\phi(t_i)))\\
&= \sum_{i=0}^{k-1} \int_{t_i}^{t_{i+1}} \nabla f(\phi(t)) \cdot \phi'(t) dt.
\end{align*}
Since 
\[
|\nabla f(\phi(t)) \cdot \phi'(t)|\le \|\nabla f(\phi(t))\| \|\phi'(t)\|\le K\|\nabla f\|_G,
\]
this completes the proof. 
\end{proof}

\section{Total variation distance}\label{tvsec}
Let $\mx$ be a Polish space equipped with  its Borel $\sigma$-algebra $\mb(\mx)$. Let $\mu$ and $\nu$ be two probability measures on $\mx$. Recall that the {\it total variation distance} between $\mu$ and $\nu$ is defined as
\begin{align*}%\label{tvdef0}
TV(\mu,\nu) := \sup_{A\in \mb(\mx)} |\mu(A)-\nu(A)|.
\end{align*}
%The total variation distance is a metric on the space of probability measures on $\mx$, meaning, in particular, that it satisfies the triangle inequality. This will be important for us later. 
The total variation distance has the following alternate representations (for proofs and discussion, see \cite[Chapter 4]{lpw09}). First, we have
\begin{align}
TV(\mu,\nu) &=\frac{1}{2} \sup\biggl\{\biggl|\int_{\mx} f(x)d\mu(x) - \int_{\mx} f(x)d\nu(x)\biggr|: \notag \\
&\qquad \qquad \text{ $f$ is a measurable map from $\mx$ into $[-1,1]$}\biggr\}.\label{tvalt}
\end{align}
Next, suppose that $\mu$ and $\nu$ are both absolutely continuous with respect to a $\sigma$-finite measure $\alpha$. Let $f$ and $g$ be the Radon--Nikodym derivatives of $\mu$ and $\nu$ with respect to $\alpha$. The total variation distance between $\mu$ and $\nu$ can be alternately expressed as
\begin{align}\label{tvform}
TV(\mu,\nu) = \frac{1}{2}\int_{\mx} |f(x)-g(x)|d\alpha(x). 
\end{align}
Finally, recall also the {\it coupling characterization} of total variation distance: If $\gamma$ is a probability measure on the product space $\mx\times \mx$, the {\it marginal probabilities} of $\gamma$ are the probability measures $\gamma_1$ and $\gamma_2$ on $\mx$ defined as
\[
\gamma_1(A) := \gamma(A\times \mx), \ \ \ \gamma_2(A) := \gamma(\mx\times A). 
\]
A {\it coupling} of $\mu$ and $\nu$ is a probability measure $\gamma$ on $\mx \times \mx$ whose marginal probabilities are $\mu$ and $\nu$. The coupling characterization of total variation distance says that
\begin{align}\label{tvinfform}
TV(\mu,\nu) = \inf\{ \gamma(\{(x,y): x\ne y\}): \gamma \text{ is a coupling of $\mu$ and $\nu$}\}. 
\end{align}
There is a standard formula for a coupling that attains the infimum in the above formula. It is as follows.  Let $\alpha$ be a $\sigma$-finite  measure such that both $\mu$ and $\nu$ are absolutely continuous with respect to $\alpha$. (For example, we can always take $\alpha = \mu+\nu$.) Let $f$ and $g$ be the Radon--Nikodym derivatives of $\mu$ and $\nu$ with respect to $\alpha$. Define 
\begin{align}\label{f1g1h}
\begin{split}
h(x) &:= \min\{f(x), g(x)\}, \\  
f_1(x) &:= f(x)-h(x), \  \ \ g_1(x) := g(x)-h(x). 
\end{split}
\end{align}
Note that $\int (f_1 - g_1)d\alpha = \int (f-g)d\alpha = 0$, and $f_1 + g_1 = |f-g|$. Thus, by \eqref{tvform},
\begin{align}\label{f1g1eq}
\int f_1(x) d\alpha (x) = \int g_1(x)d\alpha(x) = TV(\mu,\nu). 
\end{align}
%Also, note that if $TV(\mu,\nu)=0$, then $f_1$ and $g_1$ are zero almost everywhere with respect to $\alpha$. 
%Also note that if $TV(\mu,\nu)=0$, then $f_1$ and $g_1$ are zero almost everywhere with respect to $\alpha$.
For any $S\in \mb(\mx\times \mx)$, the set $\tilde{S} := \{x\in \mx: (x,x)\in S\}$ is measurable since the map $x\mapsto (x,x)$ is continuous and hence measurable.   So we can legitimately define, for any $S\in \mb(\mx\times \mx)$, 
\begin{align}\label{gammadef}
\gamma(S) := \int_{\tilde{S}} h(x) d\alpha(x) + \frac{1}{TV(\mu,\nu)}\int_{S} f_1(x) g_1(y) d\alpha^{\otimes 2}(x,y),
\end{align}
where $\alpha^{\otimes 2}$ stands for the product measure $\alpha \times \alpha$ on $\mx\times \mx$, and the second term is interpreted as zero if $TV(\mu,\nu)=0$. It is not hard to check that $\gamma$ is a probability measure on $\mx\times \mx$, and that it is a coupling of $\mu$ and $\nu$ which attains the infimum in~\eqref{tvinfform}. Also, it is not hard to check that $\gamma$ does not depend on the choice of $\alpha$. 

The reason why we took the trouble of writing down the explicit form of $\gamma$ is that we will need it in  the following lemma. It says that if $(\mu, \nu)$ and $(\mu',\nu')$ are pairs of probability measures such that $\mu$ is close to $\mu'$ and $\nu$ is close to $\nu'$ in total variation distance, then the optimal coupling of $\mu$ and $\nu$ is close to the optimal coupling $\mu'$ and $\nu'$ in total variation distance. I could not find this result in the literature, so a complete proof is given below. 
\begin{lmm}\label{tvmainlmm}
Let $\mu$, $\nu$, $\mu'$ and $\nu'$ be probability measures on a Polish space $\mx$. Let $\gamma$  be an optimal coupling of $\mu$ and $\nu$, and let $\gamma'$ be an optimal coupling of $\mu'$ and $\nu'$,  both defined according to the formula \eqref{gammadef}. Then
\[
TV(\gamma, \gamma') \le 10 \sqrt{\max\{TV(\mu,\mu'), TV(\nu, \nu')\}}. 
\]
\end{lmm}
\begin{proof}
Let $\alpha$ be a $\sigma$-finite measure on $\mx$ such that $\mu$, $\nu$, $\mu'$ and $\nu'$ are all absolutely continuous with respect to $\alpha$. Let $f$ and $g$ be the probability density functions of $\mu$ and $\nu$ with respect to $\alpha$, and define $f_1$, $g_1$ and $h$ as in \eqref{f1g1h}. Let $f'$, $g'$, $h'$, $f_1'$ and $g_1'$ be the analogous functions for $\mu'$ and $\nu'$. Define
\begin{align*}
a := \max\{TV(\mu,\nu), TV(\mu',\nu')\},  \ \ b := \max\{TV(\mu, \mu'),TV(\nu, \nu')\}.
\end{align*}
Suppose that $TV(\mu,\nu)$ and $TV(\mu',\nu')$ are both nonzero. Then for $S\in \mb(\mx\times \mx)$,
\begin{align}
&|\gamma(S)-\gamma'(S)| \notag \\
&\le \int_{\tilde{S}} |h(x)-h'(x)|d\alpha(x) \notag\\
&\qquad  + \frac{1}{TV(\mu,\nu)}\int_S |f_1(x)g_1(y) - f_1'(x)g_1'(y)| d\alpha^{\otimes 2}(x,y)\notag \\
&\qquad + \biggl|\frac{1}{TV(\mu,\nu)}-\frac{1}{TV(\mu', \nu')}\biggr|\int_S f_1'(x)g_1'(y) d\alpha^{\otimes 2}(x,y).\label{tv0}
\end{align}
Let us now estimate the three terms on the right. By the inequality
\[
|\min\{u,v\} - \min\{u',v'\}|\le |u-u'| + |v-v'|,
\]
we get
\begin{align}\label{hhprime}
|h(x)-h'(x)|\le |f(x)-f'(x)|+|g(x)-g'(x)|. 
\end{align}
By \eqref{tvform}, this gives
\begin{align}\label{tv1}
\int_{\tilde{S}} |h(x)-h'(x)|d\alpha(x)&\le 4b. 
\end{align}
Again using \eqref{hhprime} and \eqref{tvform}, note that 
\[
\int_{\mx} |f_1(x)-f_1'(x)|d\alpha(x)\le 6b, \ \ \ \int_{\mx} |g_1(x)-g_1'(x)|d\alpha(x)\le 6b.
\]
Thus, 
\begin{align}
&\int_S |f_1(x)g_1(y) - f_1'(x)g_1'(y)| d\alpha^{\otimes 2}(x,y) \notag\\
&\le \int_{\mx\times \mx} (|f_1(x)-f_1'(x)| g_1(y) + f_1'(x)|g_1(y)-g_1'(y)|) d\alpha^{\otimes 2}(x,y)\notag\\
&\le 12b. \label{tv2}
\end{align}
Next, note that by the triangle inequality for total variation distance~\cite[Remark 4.4]{lpw09},
\begin{align*}
|TV(\mu,\nu)- TV(\mu',\nu')|\le TV(\mu,\mu') + TV(\nu,\nu')\le 2b.
\end{align*}
This gives
\begin{align}
&\biggl|\frac{1}{TV(\mu,\nu)}-\frac{1}{TV(\mu',\nu')}\biggr|\int_S f_1'(x)g_1'(y) d\alpha^{\otimes 2}(x,y)\notag\\
&\le \frac{2b}{TV(\mu,\nu)TV(\mu',\nu')} \int_{\mx\times \mx}  f_1'(x)g_1'(y) d\alpha^{\otimes 2}(x,y)\notag\\
&=  \frac{2b}{TV(\mu,\nu)TV(\mu',\nu')}TV(\mu',\nu')^2\le \frac{2b}{TV(\mu,\nu)}, \label{tv3}
\end{align}
where the identity in the last line holds by \eqref{f1g1eq}. %because
%\[
%\int_{\mx} f_1'(x)d\alpha(x) = \int_{\mx} g_1'(x)d\alpha(x) = TV(\mu',\nu').
%\]
Using the bounds \eqref{tv1}, \eqref{tv2} and \eqref{tv3} in \eqref{tv0}, we get
\[
TV(\gamma,\gamma') \le 4b + \frac{14b}{TV(\mu,\nu)}. 
\]
But by the symmetry of the problem, the same bound should hold if we replace $TV(\mu,\nu)$ by $TV(\mu',\nu')$ on the right. Thus,
\begin{align}\label{tvfirst}
TV(\gamma,\gamma')\le 4b + 14b \min\biggl\{\frac{1}{TV(\mu,\nu)}, \frac{1}{TV(\mu',\nu')}\biggr\}= 4b + \frac{14 b}{a}.
\end{align}
The above bound was derived under the assumption that $TV(\mu,\nu)$ and $TV(\mu',\nu')$ are nonzero. But it is not hard to check that all the steps go through even if one or both of them are zero. Next, note that for any $S$,
\begin{align*}
\biggl|\gamma(S) - \int_{\tilde{S}} h(x)d\alpha(x)\biggr| &\le \frac{1}{TV(\mu,\nu)}\int_{\mx\times \mx} f_1(x)g_1(y)d\alpha^{\otimes 2}(x,y)\\
&= \frac{1}{TV(\mu,\nu)} TV(\mu,\nu)^2 = TV(\mu,\nu).
\end{align*}
Note that the bound holds even if $TV(\mu,\nu)=0$. Similarly, 
\[
\biggl|\gamma'(S) - \int_{\tilde{S}} h'(x)d\alpha(x)\biggr| \le TV(\mu',\nu').
\]
Thus, by \eqref{tv1},
\begin{align*}
|\gamma(S)-\gamma'(S)| &\le \int_{\tilde{S}} |h(x)-h'(x)|d\alpha(x) + TV(\mu,\nu) + TV(\mu',\nu')\\
&\le 4b + 2a. 
\end{align*}
Since this holds for any $S$, we get
\begin{align}\label{tvsecond}
TV(\gamma,\gamma') \le 4b + 2a. 
\end{align}
The proof is now completed by combining \eqref{tvfirst} and \eqref{tvsecond}, as follows. If $a < \sqrt{7b}$, we use \eqref{tvsecond} to get
\[
TV(\gamma,\gamma') \le 4b + 2\sqrt{7b} \le 10\sqrt{b}.
\]
On the other hand, if $a\ge \sqrt{7b}$, we use \eqref{tvfirst} to get
\[
TV(\gamma, \gamma') \le 4b + \frac{14b}{\sqrt{7b}}\le 10\sqrt{b}.
\]
This completes the proof of the lemma.
\end{proof}

\section{Conditional probabilities}\label{condprobsec}
Let $\mx$ and $\my$ be two Polish spaces, equipped with their respective Borel $\sigma$-algebras. A function $\phi:\mx \times \mb(\my) \to [0,1]$ is called a {\it conditional probability} from $\mx$ to $\my$ if the following two conditions are satisfied:
\begin{enumerate}
\item For each $x\in \mx$, $\phi(x,\cdot)$ is a probability measure on $\my$.
\item For each $B\in \mb(\my)$, the map $x\mapsto \phi(x,B)$ is measurable. 
\end{enumerate}
If $\alpha$ is a $\sigma$-finite measure on $\my$ such that each $\phi(x,\cdot)$ is absolutely continuous with respect to $\alpha$ with Radon--Nikodym derivative $f(x,\cdot)$, and $f$ satisfies the condition that $(x,y)\mapsto f(x,y)$ is measurable, then we say that $f(x,\cdot)$ are {\it conditional probability densities} from $\mx$ to $\my$. We will need the following lemma about conditional probabilities. Although elementary, I could not find the exact statement in the literature. 
\begin{lmm}\label{condproblmm}
Let $\mx$, $\my$ and $\phi$ be as above. Then, given any probability measure $\mu$ on $\mx$, there is a unique  probability measure $\gamma$ on $\mx\times \my$ satisfying
\begin{align}\label{gammequation}
\gamma(A\times B) = \int_A \phi(x,B)d\mu(x)
\end{align}
for any $A\in \mb(\mx)$ and $B\in \mb(\my)$. Next, suppose that we have two probability measures $\mu$ and $\mu'$ and two conditional probabilities $\phi$ and $\phi'$ having conditional probability densities $f$ and $f'$ with respect to some probability measure $\alpha$. Suppose that $f$ and $f'$ are uniformly bounded by a constant $a$. Let $\gamma$ and $\gamma'$ be  as in \eqref{gammequation}. Then
\[
TV(\gamma,\gamma') \le a TV(\mu,\mu') + \sup_{x\in \mx, y\in \my} |f(x,y)-f'(x,y)|. 
\]
\end{lmm}
\begin{proof}
The existence and uniqueness of $\gamma$ can be established by standard measure-theoretic methods, via Carath\'eodory's extension theorem~\cite[Theorem A.1.1.]{durrett10}. For the second part, we proceed as follows. Let $\beta :=\mu\times \alpha$. For each $S\in\mb( \mx\times \my)$, let
\[
\tilde{\gamma}(S) := \int_S f(x,y) d\beta(x,y). % = \int_{\mx}\int_{\my} 1_S(x,y) f(x,y)d\alpha(y)d\mu(x). 
\]
It is easy to verify that $\tilde{\gamma}$ is a probability measure on $\mx\times \my$, and that it satisfies equation~\eqref{gammequation}. So by the uniqueness of $\gamma$, we get that $\tilde{\gamma}=\gamma$. Therefore by Fubini's theorem, we have
\begin{align}\label{gammas}
\gamma(S) = \tilde{\gamma}(S) = \int_{\mx} \int_{\my} 1_S(x,y) f(x,y) d\alpha(y)d\mu(x),
\end{align}
where $1_S$ denotes the indicator function for the set $S$. A similar formula holds for $\gamma'$. Thus, if we let
\[
g(x) := \int_{\my} 1_S(x,y) f(x,y) d\alpha(y),
\]
then $g$ is measurable, and 
\begin{align}
|\gamma(S)-\gamma'(S)| &\le  \biggl|\int_{\mx} g(x)d\mu(x) - \int_{\mx} g(x) d\mu'(x)\biggr|\notag \\
&\qquad + \int_{\mx} \int_{\my} |f(x,y)-f'(x,y)| d\alpha(y)d\mu'(x). \label{cond0}
\end{align}
By assumption, there is a constant $a$ such that $f(x,y)\le a$ for all $x$ and $y$. Since $\alpha$ is a probability measure, this implies that $0\le g(x)\le a$ for all $x$. Thus, by \eqref{tvalt},
\begin{align}
\biggl|\int_{\mx} g(x)d\mu(x) - \int_{\mx} g(x) d\mu'(x)\biggr| \le a TV(\mu,\mu').\label{cond1}
\end{align}
Combining \eqref{cond0} and \eqref{cond1}, we get the required bound.
\end{proof}

\section{Lattice gauge theory in a cube}\label{infsec}
For each $N\ge 1$, let $B_N$ be the  cube $\{-N, \ldots, N\}^d$. The goal of this section is to investigate some properties of our lattice gauge theory restricted to $B_N$. Let $E_N$ be the set of positively oriented edges of $B_N$. Let $\Omega_N$ be the set of all functions from $E_N$ into $G$. Let $\partial E_N$ denote the set of positively oriented  boundary edges of $B_N$. Let $\partial \Omega_N$ denote the set of all functions from $\partial E_N$ into $G$.  Let $E_N^\circ := E_N\setminus \partial E_N$ be the set of positively oriented {\it interior edges} of $B_N$. Let $\Omega_N^\circ$ be the set of all functions from $E_N^\circ$ into $G$. Thinking of $\Omega_N$ as the Cartesian product of $\Omega_N^\circ$ and $\partial \Omega_N$, we will write an element of $\Omega_N$ as a pair $(\omega, \delta)$, where $\omega\in \Omega_N^\circ$ and $\delta\in \partial\Omega_N$. 

Let $H_N:\Omega_N \to \rr$ denote the Hamiltonian of our lattice gauge theory restricted to the cube $B_N$, viewing it as a function of a configuration $\omega \in \Omega_N^\circ$ and a boundary condition $\delta\in \partial \Omega_N$. Given a measurable function $f:\Omega_N\to \rr$ and a boundary condition $\delta\in \partial \Omega_N$, let 
\begin{align}\label{expecdef2}
\smallavg{f} := \frac{\int_{\Omega_{N}^\circ}f(\omega,\delta) e^{-\beta H_N(\omega,\delta)} d\lambda_N(\omega) }{\int_{\Omega_{N}^\circ}e^{-\beta H_N(\omega,\delta)} d\lambda_N(\omega)},
\end{align}
where $\lambda_N$ is the normalized product Haar measure on $\Omega_N^\circ$, provided that the numerator is well-defined. Note that this is a function of $\delta$.
%Note that this is a function of $\delta$. If $f$ has no dependence on $\delta$, then this is  just the expected value of $f$ in the usual sense, but if $f$ has some dependence on $\delta$, then it is a generalization of \eqref{expecdef}. 

%Let us generalize the notion of a local function  by saying that $f:\Omega_N\to\rr$ is a {\it generalized local function} supported on an edge $e\in E_N$ if $f(\omega,\delta)$ depends only on those  $\omega_u$ and $\delta_v$ such that $u$ and $v$ are adjacent to $e$ (that is, they share some plaquette with $e$). We will mainly be interested in the following generalized local function. 

Take any edge $e\in \partial E_N$, and consider the Hamiltonian $H_N$ as a function of only $\delta_e$, fixing the matrices assigned to all other edges. Let $\nabla_eH_N$ denote the gradient of this function (identifying $M_n(\cc)$ with $\rr^{2n^2}$, as before). Then it is easy to see that each component of $\nabla_e H_N$ is a local function supported on $e$. We will abbreviate this by simply saying that $\nabla_e H_N$ is a local function supported on $e$. Moreover, observe that $\|\nabla_e H\|$ can be bounded by a constant that depends only on $G$ and $d$.

%It is simple to see that the strong mass gap assumption continues to hold for generalized local functions, because when we are computing expectations, we fix the boundary condition; and when we fix the boundary condition, generalized local functions become local functions. 

Let $f:\Omega_N\to\rr$ be a bounded measurable function. Let $e\in \partial E_N$ be an edge such that $f$ has no dependence on $\delta_e$. Then using the dominated convergence theorem for Lebesgue integrals, it is not hard to see that the gradient can be moved inside the integrals in the following calculation, giving
\begin{align}
\nabla_e \smallavg{f} &= \frac{\int_{\Omega_{N}^\circ}f(\omega,\delta ) \nabla_e(e^{-\beta H_N(\omega,\delta)}) d\lambda_N(\omega) }{\int_{\Omega_{N}^\circ}e^{-\beta H_N(\omega,\delta)} d\lambda_N(\omega)} \notag\\
&\qquad - \frac{\int_{\Omega_{N}^\circ }f(\omega,\delta) e^{-\beta H_N(\omega,\delta)} d\lambda_N(\omega) \int_{\Omega_{N}^\circ}\nabla_e(e^{-\beta H_N(\omega,\delta)}) d\lambda_N(\omega)}{(\int_{\Omega_{N}^\circ}e^{-\beta H_N(\omega,\delta)} d\lambda_N(\omega))^2}\notag \\
&= -\beta \smallavg{f\nabla_e H_N} +\beta \smallavg{f}\smallavg{\nabla_eH_N}, \label{covform}
\end{align}
where the last identity holds because 
\[
\nabla_e (e^{-\beta H_N(\omega,\delta)}) = -\beta e^{-\beta H_N(\omega,\delta)}\nabla_e H_N(\omega,\delta).
\]
Note that $f \nabla_e H_N$ and $\nabla_e H_N$ are matrix-valued functions, so the expectations written above are to be interpreted as matrices of expected values.
%We will make considerable use of \eqref{covform}. Now, from the formula \eqref{hndef}, it is clear that each component of $\nabla_eH(\omega)$ is a local function supported on $e$. We will simply say that $\nabla_e H$ is a local function. This fact will also be very useful for us. %We will frequently use the above observations about $\nabla_e H_N$.

%\section{Truncating the boundary influence}\label{truncsec}
%The expected value of a generalized local function $f:\Omega_N\to \rr$, as defined in equation \eqref{expecdef2}, is a function of the boundary condition. Using the assumption, we can show that this function can be approximated by a function of the boundary values that are close to the supporting edge of $f$. The following lemma makes this precise. It is a consequence of equation \eqref{covform} and Corollary~\ref{liecor}. 
%Take any local function $f:\Omega_N \to [-1,1]$ and consider the expected value $\smallavg{f}$ as a function of the boundary condition $\delta$. %For each $r\ge 1$, let $\partial E(f,r)$ be the set of edges in $\partial E_N$ that are within distance $r$ from the supporting edge of~$f$. 

%Using the above formula and Corollary \ref{liecor}, we now prove the following lemma, which shows that the expected value of a local function supported on a boundary edge is approximately equal to a function of the boundary configuration near that edge. 

In the statement of the following lemma and in all that follows, we will use $C, C_1,C_2,\ldots$ to denote positive constants that may depend on $G$, $\beta$, $d$ and $\pi$, and nothing else. The values of these constants may change from line to line.
\begin{lmm}\label{bdrylmm}
Let $f:\Omega_N \to[-1,1]$ be a local function supported on an edge $e\in \partial E_N$. Take any $r\ge 1$. Then there is a function $g:\partial \Omega_N \to \rr$ that depends only on $\{\delta_u: u\in \partial E_N, \, \dist(e,u)\le r\}$, such that for any boundary condition $\delta$, $|\smallavg{f} - g(\delta)|\le C_1 N^{d-1}e^{-C_2r}$. 
%where $C_1$ and $C_2$ depend only on $G$, $\beta$ and $d$.
\end{lmm}
\begin{proof}
Take any $u\in \partial E_N$ such that $\dist(e,u)>r$. Then $u$ is not adjacent to $e$, meaning that they do not belong to a common plaquette. Since $r\ge 1$ and $f$ is a  local function supported on $e$, this shows that $f$ has no dependence of $\delta_u$. So, by the formula \eqref{covform} and the fact that $f$ and $\nabla_u H_N$ are local functions whose magnitudes are bounded by constants that depend only on $G$ and $d$, and whose supporting edges are separated by a distance greater than $r$, the correlation decay assumption implies that
\begin{align*}
\|\nabla_u\smallavg{f}\| &= |\beta| \|\smallavg{f\nabla_u H_N} - \smallavg{f}\smallavg{\nabla_uH_N}\|\le C_1 e^{-C_2r}.
\end{align*}
Note that this bound holds irrespective of the boundary condition. Therefore, Corollary \ref{liecor} implies that if we replace $\delta_u$ by $I$, the value of $\smallavg{f}$ changes by at most $C_1e^{-C_2r}$. We can perform this operation successively to replace $\delta_u$ by $I$ for each $u\in \partial  E_N$ such that $\dist(e,u)>r$. The total change in $\smallavg{f}$ will be bounded by $C_1 N^{d-1} e^{-C_2 r}$. The new value of $\smallavg{f}$ is a function of $\{\delta_u: u\in \partial E_N,\, \dist(e,u)\le r\}$. Let this function be denoted by $g$. Clearly, this $g$ has the desired property.
\end{proof}

%\section{The $r$-neighborhood of a boundary edge}\label{rnbhdsec}
Take any $e\in \partial E_N$ and any $1\le r \le N/4$. Let $x = (x_1,\ldots,x_d)$ and $y = (y_1,\ldots,y_d)$ be the endpoints of $e$, in lexicographic order. For each $i$, let
\begin{align*}
a_i :=
\begin{cases}
x_i - 2r &\text{ if } -N+2r\le x_i\le N-2r,\\
-N &\text{ if } x_i < -N+2r,\\
N-4r &\text{ if } x_i > N-2r.
\end{cases}
\end{align*}
Let $b_i := a_i + 4r$. Let 
\[
B(e,r) := ([a_1,b_1]\times \cdots \times [a_d,b_d]) \cap \zz^d.
\]
Clearly, $B(e,r)$ is a cube. Let $E(e,r)$ denote the set of positively oriented edges of $B(e,r)$ and let $\partial E(e,r)$ denote the positively oriented boundary edges of $B(e,r)$. We will refer to $B(e,r)$ as the {\it $r$-neighborhood} of the edge  $e$ (see Figure \ref{rnbhpic}). The following lemma gathers some basic facts about the structure of $B(e,r)$. 

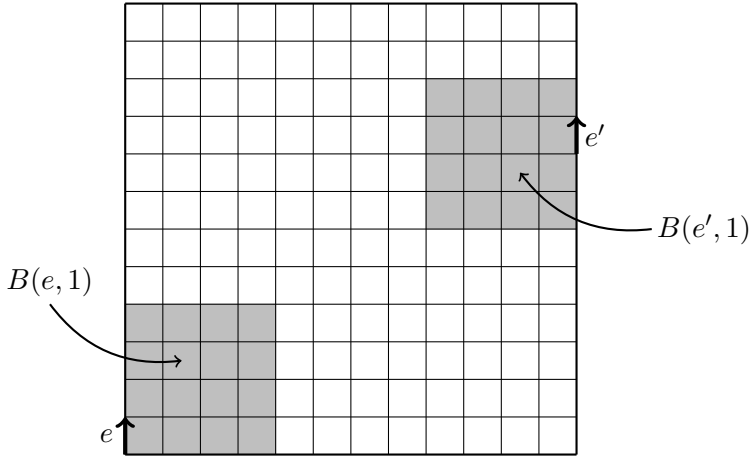
\begin{figure}[t!]
\begin{center}
\begin{tikzpicture}[scale = 1]
\draw[lightgray, fill] (0,0) rectangle (2,2);
\draw[lightgray, fill] (4,3) rectangle (6,5);
\draw[thick] (0,0) to (6,0);
\draw[thick] (6,0) to (6,6);
\draw[thick] (6,6) to (0,6);
\draw[thick] (0,6) to (0,0);
\draw (.5,0) to (.5,6);
\draw (1,0) to (1,6);
\draw (1.5,0) to (1.5,6);
\draw (2,0) to (2,6);
\draw (2.5,0) to (2.5,6);
\draw (3,0) to (3,6);
\draw (3.5,0) to (3.5,6);
\draw (4,0) to (4,6);
\draw (4.5,0) to (4.5,6);
\draw (5,0) to (5,6);
\draw (5.5,0) to (5.5,6);
\draw (0,.5) to (6,.5);
\draw (0,1) to (6,1);
\draw (0,1.5) to (6,1.5);
\draw (0,2) to (6,2);
\draw (0,2.5) to (6,2.5);
\draw (0,3) to (6,3);
\draw (0,3.5) to (6,3.5);
\draw (0,4) to (6,4);
\draw (0,4.5) to (6,4.5);
\draw (0,5) to (6,5);
\draw (0,5.5) to (6,5.5);
%\draw [thick, decorate, decoration = {brace}] (-.1,0) to (-.1,.5);
\node at (-.25,.25) {$e$};
\draw[ultra thick, ->] (0,0) to (0,.5);
\node at (6.25,4.25) {$e'$};
\draw[ultra thick, ->] (6,4) to (6,4.5);
\draw[thick, ->] (-1,2) to[bend right = 30] (.75,1.25);
\node at (-1,2.3) {$B(e,1)$};
\draw[thick, ->] (7,3) to[bend left = 30] (5.25,3.75);
\node at (7.7,3) {$B(e',1)$};
\end{tikzpicture}
\caption{The $1$-neighborhoods of two boundary edges $e$ and $e'$.\label{rnbhpic}}
\end{center}
\end{figure}

\begin{lmm}\label{neighborlmm}
The set $B(e,r)$ defined above is a cube of side-length $4r$, and is a subset of $B_N$. The edge $e$ is a boundary edge of $B(e,r)$. Any edge of $B_N$ that is adjacent to $e$ is also an edge of $B(e,r)$. Lastly, any $u\in \partial E(e,r)\setminus \partial E_N$ must satisfy  $\dist(e,u)> r$.
\end{lmm}
\begin{proof}
For simplicity, let us write $B$ and $E$ instead of $B(e,r)$ and $E(e,r)$. By construction, $a_i$ and $b_i$ are in $[-N,N]$ for each $i$. Therefore, $B\subseteq B_N$. Also by construction, $b_i-a_i=4r$ for each $i$. Thus, $B$ is a cube of side-length $4r$. 

By construction, $x_i\in [a_i,b_i]$ for each $i$. Thus, $x\in B$. Since $x$ and $y$ are neighbors and $y$ is lexicographically bigger than $x$, there is exactly one coordinate $i$ such that $y_i=x_i+1$, and $y_j=x_j$ for all $j\ne i$. Since $y\in B_N$, it must be that $x_i <N$, and hence $y_i\in [a_i, b_i]$. For $j\ne i$, $y_j=x_j\in [a_j, b_j]$. So we conclude that $y\in B$. Since $x$ and $y$ are both in $B$, we get that $e\in E$. But $B$ is a cube contained in $B_N$ and $e\in \partial E_N$. Therefore $e$ must be a boundary edge of $B$. Thus, $e\in \partial E$.

Take any $u\in E_N\setminus E$.  Let $w =(w_1,\ldots,w_d)$ and $z= (z_1,\ldots,z_d)$ be the endpoints of $u$, in lexicographic order. Since $u\notin E$, at least one of $w$ and $z$ is not in $B$. Suppose that $w\notin B$. Then there is a coordinate $j$ such that $w_j\notin [a_j,b_j]$. That is, either $w_j<a_j$ or $w_j >b_j$. Since $w\in B_N$, the definitions of $a_j$ and $b_j$ now show that $|w_j-x_j| > 2r$. Since $|x_j-y_j|\le 1$ and $|w_j-z_j|\le 1$, and $r\ge1$, this gives
\begin{align*}
\dist(e,u)^2 &\ge \biggl(\frac{w_j+z_j}{2} - \frac{x_j+y_j}{2}\biggr)^2>  (2r-1)^2 \ge r^2.
\end{align*}
A similar argument shows that if $z\notin B$, then $\dist(e,u)> r$. Thus, we conclude that if $u\in E_N\setminus E$, then $\dist(e,u) > r\ge 1$. In particular, $u$ is not adjacent to $e$. 

%Then each $w_j$ can either be $x_j$, or $x_{j}+1$, or $x_j-1$, and the same goes for each $z_j$. Considering all possibilities, it is now easy to argue that $w$ and $z$ are both in $B(e,r)$, and so $u\in E(e,r)$. 
Finally, take any $u\in \partial E\setminus \partial E_N$.  Let $w =(w_1,\ldots,w_d)$ and $z= (z_1,\ldots,z_d)$ be the endpoints of $u$, in lexicographic order. Since $u\in \partial E$, there must exist $j$ such that either $w_j=z_j=a_j$ or $w_j=z_j=b_j$. Moreover, since $u\notin \partial E_N$, there must exist $j$ with this property and also satisfying $-N<w_j<N$. Take any such $j$. Then it is easy to see that $|w_j-x_j|\ge 2r$. Since $|y_j-x_j|\le 1$ and $r\ge1$, %this gives
\begin{align*}
\dist(e,u)^2 &\ge \biggl(\frac{w_j+z_j}{2} - \frac{x_j+y_j}{2}\biggr)^2\\
&=  \biggl(w_j - \frac{x_j+y_j}{2}\biggr)^2\ge \biggl(2r-\frac{1}{2}\biggr)^2 > r^2.
\end{align*}
This completes the proof of the lemma.
\end{proof}

%\section{Influence outside the $r$-neighborhood}\label{rnbhinfsec}
%The goal of this section is to establish that a boundary edge has very little influence outside its $r$-neighborhood if $r$ is sufficiently large. More generally, any set of boundary edges has little influence outside the union of the $r$-neighborhoods of the elements of that set. This is a consequence of Lemma \ref{bdrylmm}, Corollary \ref{liecor} and Lemma \ref{neighborlmm}.

Take any $e\in \partial E_N$ and any $1\le r\le N/4$. Let $B(e,r)$, $E(e,r)$ and $\partial E(e,r)$ be defined as above. The following lemma gives an upper bound on the correlation between a local function supported on $e$ and a function that depends only on edges outside the $r$-neighborhood of $e$. 
\begin{lmm}\label{influencelmm}
Let $e$ and $r$ be as above. Let $f:\Omega_N^\circ \to [-1,1]$  be a measurable function that depends only on $\{\omega_u: u\in E_N^\circ \setminus E(e,r)\}$. Let $h:\Omega_N\to[-1,1]$ be a local function supported on $e$. Then under any boundary condition on $B_N$, 
\[
|\smallavg{fh} - \smallavg{f}\smallavg{h}|\le C_1  r^{d-1} e^{-C_2r}.
\]
%where $C_1$ and $C_2$ are positive constants that depend only on $G$, $\beta$ and $d$.
\end{lmm}
\begin{proof}
%Replacing $f$ by $f/\|f\|_\infty$, let us assume that $\|f\|_\infty\le 1$. 
As before, let us write $B$ and $E$ instead of $B(e,r)$ and $E(e,r)$. Fix some boundary condition $\delta$ and let $\mu$ be the probability measure defined by our lattice gauge theory in $B_N$ with this boundary condition. Let $\mu'$ denote the conditional probability measure given $\{\omega_u: u\in E_N^\circ \setminus E^\circ\}$, where $E^\circ$ is the set of positively oriented interior edges of $B$. Since a lattice gauge theory is a Markov random field, it is clear that $\mu'$ is again a lattice gauge theory on the cube $B$, with boundary condition $\delta'$ on $\partial E$, where 
\[
\delta'_u = 
\begin{cases}
\omega_u &\text{ if } u\in \partial E \setminus \partial E_N,\\
\delta_u &\text{ if } u\in \partial E \cap \partial E_N.
\end{cases}
\] 
Let $\smallavg{h}'$ denote the expected value of $h$ with respect to $\mu'$.  By Lemma \ref{neighborlmm}, any edge of $B_N$ that is adjacent to $e$ is also an edge of $B$. Therefore by Lemma~\ref{bdrylmm} applied to the cube $B$, there is a function $g(\delta')$ of the boundary condition $\delta'$, which depends only on $\{\delta'_u: u\in \partial E, \  \dist(e,u)\le r\}$, such that  
\[
|\smallavg{h}' - g(\delta')| \le C_1 r^{d-1} e^{-C_2r}.
\]
But for any $u\in \partial E \setminus \partial E_N$, Lemma \ref{neighborlmm} tells us that $\dist(e,u) >r$.  Thus, $g(\delta')$ depends only on  a subset of $\{\delta_u': u\in \partial E\cap \partial E_N\}$. In particular, $g(\delta')$  is simply a function of the original boundary condition $\delta$ and has no dependence on $\{\omega_u: u\in \partial E\setminus \partial E_N\}$. So we can write $g(\delta)$ instead of $g(\delta')$. 

Now note that by the tower property of conditional expectation and the fact that $f$ has no dependence on $\{\omega_u: u\in E_N^\circ\cap E\}$, we get  $\smallavg{fh} = \smallavg{f \smallavg{h}'}$. Since $f$ maps into $[-1,1]$, this gives 
\begin{align*}
|\smallavg{fh} - \smallavg{f}g(\delta)| &= |\smallavg{f (\smallavg{h}' - g(\delta))}|\\
&\le \smallavg{|\smallavg{h}' - g(\delta)|}\le C_1 r^{d-1}e^{-C_2 r}. 
\end{align*}
Similarly, since $\smallavg{h} = \smallavg{\smallavg{h}'}$, 
\begin{align*}
|\smallavg{f}\smallavg{h} - \smallavg{f}g(\delta)| &\le  |\smallavg{\smallavg{h}' - g(\delta)}|\\
&\le \smallavg{|\smallavg{h}' - g(\delta)|}\le C_1 r^{d-1}e^{-C_2 r}.
\end{align*}
The claim now follows by combining the above inequalities.
\end{proof}
Combining the above lemma with Corollary \ref{liecor} and equation \eqref{covform}, we obtain the following upper bound. It says that the expected value of a function that depends only on edges outside the $r$-neighborhood of $e$ cannot change by much if $\delta_e$ is replaced by a new value. 
\begin{cor}\label{infcor}
Let $f$ and $e$ be as in Lemma \ref{influencelmm}. Then 
\[
\|\nabla_e \smallavg{f}\|\le C_1 r^{d-1} e^{-C_2r}.
\]
Consequently, if $\delta_e$ is replaced by any other value $\delta_e'$, the value of $\smallavg{f}$ changes by at most $C_1  r^{d-1} e^{-C_2r}$.
\end{cor}
\begin{proof}
Recall that by \eqref{covform}, $\nabla_e\smallavg{f} = -\beta \smallavg{f \nabla_e H_N} +\beta \smallavg{f}\smallavg{\nabla_e H_N}$. Also, recall from the discussion in Section \ref{infsec} that $\nabla_e H_N$ is a local function supported on $e$, whose norm is bounded by a constant that depends only on $G$ and $d$. These two facts, combined with Lemma \ref{influencelmm}, prove the first claim. For the second, apply the first claim and Corollary \ref{liecor}.
\end{proof}
We will now extend Corollary \ref{infcor} to collections of boundary edges. Let $A$ be any nonempty subset of $\partial E_N$. Define $E(A,r)$ to be the union of $E(e,r)$ over all $e\in A$. 
\begin{lmm}\label{rnbhlmm}
Let $E(A,r)$ be as above. Take any measurable function $f:\Omega_N^\circ \to [-1,1]$ that depends only on $\{\omega_u: u\in E_N^\circ\setminus E(A,r)\}$. Then for any two boundary conditions $\delta$ and $\delta'$ that agree on the complement of $A$, the values of $\smallavg{f}$ differ by at most $C_1|A|r^{d-1}e^{-C_2r}$. %, where $C_1$ and $C_2$ depend only on $G$, $\beta$ and $d$.
\end{lmm}
\begin{proof}
By Corollary \ref{infcor}, $\smallavg{f}$ changes by at most $C_1r^{d-1}e^{-C_2 r}$ when the value of $\delta_e$ is changed for a single $e\in A$. Therefore the cumulative change when $\delta_e$ is changed for all $e\in A$ is at most $C_1|A|r^{d-1}e^{-C_2 r}$.
\end{proof}

%\section{Influence of the spatial boundary on the $r$-core}\label{rcoresec}
%Take any $N\ge 1$ and $1\le r\le N/2$. Let $A$ be a nonempty subset of $\partial E_N$, and as in Section \ref{rnbhinfsec}, let $E(A,r)$ be the union of $E(e,r)$ over all $e\in A$. 

Take any nonempty set  $A\subseteq \partial E_N$. Let $\Omega_{A,r}$ be the set of all functions from $E_N^\circ \setminus E(A,r)$  into $G$, and let $\Omega_{A,r}^*$ be the set of all functions from  $E_N^\circ \cap E(A,r)$ into $G$. Note that $\Omega_N^\circ$ can be viewed as the Cartesian product of $\Omega_{A,r}$ and $\Omega_{A,r}^*$. The following corollary of Lemma \ref{rnbhlmm} tells us that if $\mu$ is the probability measure defined by our lattice gauge theory in $B_N$, then the  marginal probability of $\mu$ on $\Omega_{A,r}$ has a  very small dependence on $\{\delta_e:e\in A\}$. This is quantified by a bound in total variation distance. 
%draw picture?
%Recall the definition of the $r$-core $B_{N,r}$ of the cube $B_N$ and the definitions of temporal and spatial faces of a cube from Section \ref{notationsec}. Recall also  the definitions of $E_{N,r}$ and $\Omega_{N,r}$. Let $E_{N,r}^* := E_N^\circ \setminus E_{N,r}$ and let $\Omega_{N,r}^*$ denote the set of all functions from $E_{N,r}^*$ into $G$. Then $\Omega_N^\circ$ can be viewed as the cartesian product of $\Omega_{N,r}$ and $\Omega_{N,r}^*$.  The following lemma quantifies the influence of the spatial boundary edges on the $r$-core.  
\begin{cor}\label{rnbhcor}
Let $\delta$ and $\delta'$ be two boundary conditions on $B_N$, defining two probability measures $\mu$ and $\mu'$ according to our lattice gauge theory. Let $A$ be the set of all $e\in \partial E_N$ such that $\delta_e\ne \delta_e'$.  Viewing $\Omega_N^\circ$ as $\Omega_{A,r}\times \Omega_{A,r}^*$ (defined above), let $\mu_r$ and $\mu_r'$ be the marginal probabilities of $\mu$ and $\mu'$ on $\Omega_{A,r}$. Then 
\[
TV(\mu_r, \mu_r')\le C_1 |A| r^{d-1}e^{-C_2r}.
\]
%where $C_1$ and $C_2$ are positive constants that depend only on $G$, $\beta$ and $d$. 
\end{cor}
\begin{proof}
Let $f:\Omega_{A,r}\to [-1,1]$ be a measurable function. Then $f$ has a natural extension to a function on $\Omega_N^\circ$, which we also denote by $f$. By Lemma \ref{rnbhlmm}, $\smallavg{f}$ changes by at most $C_1 |A| r^{d-1}e^{-C_2r}$ if the boundary condition $\delta$ is replaced by $\delta'$. The result now follows by the formula \eqref{tvalt} for total variation distance.
\end{proof}

\section{Coupling in a cube}\label{cubecouplingsec}
Take any $N\ge 1$ and $1\le r\le N/4$. In this section we give a general prescription for coupling two lattice gauge theories with two different boundary conditions on $B_N$. This is the first step towards defining a similar coupling on a slab, which is the main component of the proof of Theorem \ref{main2}. 

Let us endow the space of all pairs of boundary conditions on $B_N$ with the usual Euclidean metric (viewing it as a subset of a Euclidean space of sufficiently large dimension). Also, let us endow the space of all probability measures on $\Omega_N^\circ \times \Omega_N^\circ$ with the total variation metric. These topologies generate Borel $\sigma$-algebras on the two spaces. We will say that a map from one space to the other is measurable if it  is measurable with respect to these $\sigma$-algebras. 
\begin{lmm}\label{cubecoupling}
Let $\delta$ and $\delta'$ be two boundary conditions on $B_N$ and let $\mu$ and $\mu'$ be the probability measures defined by our lattice gauge theory with these boundary conditions. Let $A$ be the set of all $e\in \partial E_N$ such that $\delta_e\ne \delta_e'$. Let $E(A,r)$ be the union of $E(e,r)$ over all $e\in A$, where $E(e,r)$ is the $r$-neighborhood defined in the previous section.  Then there is a coupling $\gamma$ of $\mu$ and $\mu'$ such that 
\begin{align}\label{gammaexist}
\begin{split}
&\gamma(\{(\omega,\omega'): \omega_u \ne \omega_u' \textup{ for some } u\in E_N^\circ\setminus E(A,r)\}) \\
&\qquad \qquad \qquad\qquad \qquad \le C_1|A|  r^{d-1}e^{-C_2r}.
\end{split}
\end{align}
Moreover, the coupling can be defined in such a way that the map $(\delta,\delta')\mapsto \gamma$ is measurable in the sense defined above.
\end{lmm}
\begin{proof}
As in Corollary \ref{rnbhcor}, let us view $\Omega_N^\circ$ as $\Omega_{A,r}\times \Omega_{A,r}^*$. Let $\mu_r$ and $\mu_r'$ be the marginal probabilities of $\mu$ and $\mu'$ on $\Omega_{A,r}$. By Corollary \ref{rnbhcor} and the coupling characterization of total variation distance, there is a coupling $\gamma_r$ of $\mu_r$ and $\mu_r'$ satisfying
\begin{align}\label{gammar}
\gamma_r(\{(x,x')\in \Omega_{A,r}\times \Omega_{A,r}:x\ne x'\})  \le C_1 |A| r^{d-1}e^{-C_2r}. 
\end{align}
Now, given any boundary condition $\delta$, our lattice gauge theory defines a conditional probability $\phi_\delta$ from $\Omega_{A,r}$ to $\Omega_{A,r}^*$.  For $x\in \Omega_{A,r}$ and $y\in \Omega_{A,r}^*$, let $f_\delta(x,y)$ denote the conditional probability density (with respect to normalized product Haar measure) of $\phi_\delta$ at the point $(x,y)$. Then the function $g_{\delta,\delta'}(x,x',y,y') := f_\delta(x,y)f_{\delta'}(x',y')$ 
is a conditional probability density from $\Omega_{A,r}\times \Omega_{A,r}$ to $\Omega_{A,r}^*\times \Omega_{A,r}^*$, corresponding to the conditional probability $\phi_{\delta,\delta'}((x,x'),\cdot) := \phi_\delta(x,\cdot )\times \phi_{\delta'}(x',\cdot)$.  
Using the first part of Lemma \ref{condproblmm} with these conditional probability densities, we can now extend $\gamma_r$ to a probability measure $\gamma$ on $\Omega_N^\circ \times \Omega_N^\circ$. By the definition of $\gamma$ and the property~\eqref{gammar} of $\gamma_r$, we see that  $\gamma$ satisfies \eqref{gammaexist}. Let us now verify that $\gamma$ is a coupling of $\mu$ and $\mu'$. Take any $S\in \mb(\Omega_{N}^\circ)$ of the form $S_1\times S_2$, where $S_1\in \mb(\Omega_{A,r})$ and $S_2 \in \mb(\Omega_{A,r}^*)$. Then 
\begin{align*}
\gamma(S_1\times \Omega_{A,r}\times S_2\times \Omega_{A,r}^* ) &= \int_{S_1\times \Omega_{A,r}} \phi_\delta(x,S_2)\phi_{\delta'}(x',\Omega_{A,r}^*) d\gamma_r(x,x')\\
&= \int_{S_1\times \Omega_{A,r}} \phi_\delta(x,S_2) d\gamma_r(x,x')\\
&= \int_{S_1} \phi_\delta(x,S_2)d\mu_r(x) = \mu(S_1\times S_2). 
\end{align*}
Similarly, $\gamma(\Omega_{A,r}\times S_1\times \Omega_{A,r}^*\times  S_2 ) = \mu'(S_1\times S_2)$. From this, it is easy to see using the uniqueness part of Carath\'eodory's extension theorem that  $\gamma$ is indeed a coupling of $\mu$ and $\mu'$. 

Finally, to prove measurability of $(\delta,\delta')\mapsto \gamma$, we argue as follows. For any $A\subseteq \partial E_N$, let $\Delta(A)$ be the set of all pairs of boundary conditions $(\delta, \delta')$ such that $A = \{e: \delta_e\ne \delta_e'\}$. The space $\Delta$ of all pairs of boundary conditions is the disjoint union of these sets, and each $\Delta(A)$ is a measurable subset of $\Delta$. So it suffices to prove that $(\delta,\delta')\mapsto \gamma$ is measurable on each $\Delta(A)$. We will, in fact, show that this map is continuous on every $\Delta(A)$. 

So take any $A\subseteq \partial E_N$. First, suppose that $A$ is nonempty. It is not hard to verify directly from the definition of lattice gauge theory that $\delta\mapsto \mu$ is a continuous map, and therefore so is $(\delta,\delta')\mapsto (\mu,\mu')$. The definition of total variation distance makes it clear that $\mu\mapsto \mu_r$ is continuous, and therefore so is the map $(\mu,\mu')\mapsto (\mu_r, \mu_r')$. By the inequality from Lemma \ref{tvmainlmm}, $(\mu_r,\mu_r')\mapsto \gamma_r$ is continuous. Combining, we get that $(\delta,\delta')\mapsto \gamma_r$ is a continuous map on $\Delta(A)$.

Now, it is not hard to see that $(\delta,x,y)\mapsto f_\delta(x,y)$ is a continuous map, and therefore uniformly continuous since its domain is compact. Consequently, the map $(\delta,x,y,\delta',x',y') \mapsto g_{\delta,\delta'}(x,x',y,y')$ is uniformly continuous. Thus, by the second part of Lemma \ref{condproblmm}, $(\delta, \delta', \gamma_r) \mapsto \gamma$ is continuous. But as observed above, $\gamma_r$ is itself a continuous function of $(\delta,\delta')$. Thus, $(\delta,\delta')\mapsto \gamma$ is continuous. 

If $A$ is empty, the proof is simpler. In this case, $\gamma_r = \gamma$, so we already have the continuity of $(\delta,\delta')\mapsto \gamma$ from the first step above. 
\end{proof}

\section{Local update map on a slab}\label{localsec}
Given any positive integers $M$ and $N$, define the slab
\[
S_{M,N} := \{-N,\ldots,N\}\times \{-M, \ldots,M\}^{d-1}. % ([-N,N]\times [-M,M]^{d-1})\cap \zz^d.
\]
The faces of $S_{M,N}$ corresponding to $x_1=N$ and $x_1=-N$ will be called the `temporal' faces (because the first coordinate denotes time), and the remaining part of the boundary will be called the `spatial boundary' of $S_{M,N}$. %We will say that the {\it thickness}  or {\it height} of the slab is $2N$ and the {\it width} of the slab is $2M$.

Let $E_{M,N}$  be the set of positively oriented edges of $S_{M,N}$, $E_{M,N}^\circ$ be the set of positively oriented interior edges, and $\partial E_{M,N}$ to be the set of positively oriented boundary edges. The spatial boundary will be denoted by $\partial' E_{M,N}$. The set of all maps from $E_{M,N}^\circ$ into $G$ will be denoted by $\Omega_{M,N}^\circ$, and the set of all maps from $\partial E_{M,N}$ into $G$ by $\partial \Omega_{M,N}$. %Lattice gauge theory on $S_{M,N}$ is defined analogously.

Take any $M$ and $N$, and let $\delta$ and $\delta'$ be two boundary conditions on $S_{M,N}$ that agree on the temporal faces. Let $\mu$ and $\mu'$ be the probability measures defined by these two boundary conditions, according to our lattice gauge theory. Take any $1\le r\le N/4$. We will now define the {\it local update map}, which maps any coupling of $\mu$ and $\mu'$ to a `better' coupling of $\mu$ and $\mu'$.  This is the second step towards the construction of a coupling in a slab. %, according to the following prescription. %The construction of the local update map uses the coupling in a cube defined in the previous section. The purpose of the local update map is to `improve' any given coupling by making the marginal probabilities more tightly coupled. 

Take any translate $B$ of the cube $B_N$ that is contained in $S_{M,N}$, such that $B$ has no intersection with the spatial boundary of $S_{M,N}$. (Assume that $M > N$, so that such a $B$ exists.) Then $B$ must be of the form
\[
B = ([-N, N] \times [a_2,b_2]\times \cdots \times [a_d,b_d])\cap \zz^d,
\]
where $a_i$ and $b_i$ are integers such that $-M< a_i< b_i< M$ and $b_i-a_i = 2N$ for each $i$. Let $\fb$ be the set of all such $B$. (See Figure \ref{cubeslabpic}.)

% Note that $\Omega_{M,N}^\circ$ can be viewed as the Cartesian product of $\Omega_B$ and $\Omega_B^*$. 
%draw picture?

\begin{figure}[t!]
\begin{center}
\begin{tikzpicture}[scale = 1]
\draw[lightgray, fill] (3,0) rectangle (6,3);
\draw[thick] (0,0) rectangle (10,3);
\draw[thick] (3,0) rectangle (6,3);
\node at (4.5, 1.5) {$B$};
\draw [thick, decorate, decoration = {brace, amplitude = 7pt}] (-.1,0) to (-.1,3);
\draw [thick, decorate, decoration = {brace, mirror, amplitude = 10pt}] (0,-.1) to (10,-.1);
\node at (-.7, 1.5) {$2N$};
\node at (5, -.7) {$2M$};
\end{tikzpicture}
\caption{A cube of width $2N$ sitting inside the slab $S_{M,N}$.\label{cubeslabpic}}
\end{center}
\end{figure}
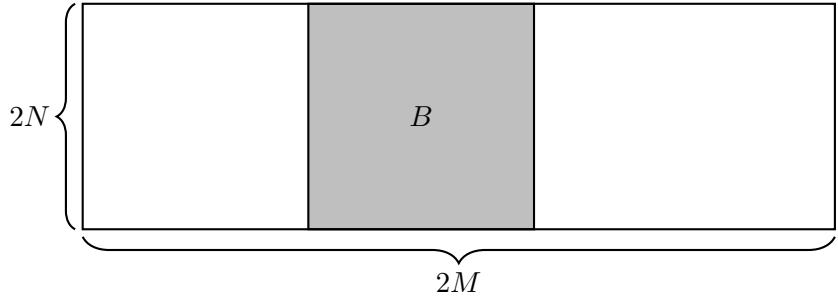

%(Please note that $\Omega_B$ is {\it not} the set of functions from the edges in $B$ into $G$, but the set of functions from the edges {\it outside} $B$ into $G$. We adopt this potentially confusing notation for consistency with certain conventions later.) 

Let $E^\circ$ be the set of interior edges  of $B$. Let $\Omega_B$ be the set of all functions from $E_{M,N}^\circ \setminus E^\circ$ into $G$, and let $\Omega_B^*$ be the set of all functions from $E^\circ$ into $G$. Let $\gamma$ be a coupling of $\mu$ and $\mu'$. Viewing $\Omega_{M,N}^\circ \times \Omega_{M,N}^\circ$ as $\Omega_B \times \Omega_B\times \Omega_B^* \times\Omega_B^*$, let $\gamma_B$ be the marginal probability of $\gamma$ on $\Omega_B\times \Omega_B$. 

Let $\mu_B$ and $\mu'_B$ be the marginal probabilities of $\mu$ and $\mu'$ on $\Omega_B$, viewing $\Omega_N^\circ$ as $\Omega_B \times \Omega_B^*$. We claim that $\mu_B$ and $\mu'_B$ are also the marginal probabilities of the probability measure $\gamma_B$. To see this, take any $S\in \mb(\Omega_B)$. Then
\begin{align*}
\gamma_B(S\times \Omega_B) &= \gamma(S\times \Omega_B \times \Omega_B^*\times \Omega_B^*)\\
&= \mu(S\times \Omega_B^*) \ \ \ \text{(since $\gamma$ is a coupling of $\mu$ and $\mu'$)}\\
&= \mu_B(S), 
\end{align*}
and similarly, $\gamma_B(\Omega_B \times S) = \mu'_B(S)$.

Now, any pair $(x,x')\in \Omega_B\times \Omega_B$ defines a pair of boundary conditions on $B$. By Lemma \ref{cubecoupling}, this pair of boundary conditions  lets us define a probability measure $\phi((x,x'),\cdot)$ on $\Omega_B^*\times \Omega_B^*$, which is a coupling of the two lattice gauge theories $\mu_x$ and $\mu_{x'}$ defined by the two boundary conditions. By the measurability assertion of Lemma \ref{cubecoupling}, $\phi$ is a conditional probability from $\Omega_B\times \Omega_B$ to $\Omega_B^*\times \Omega_B^*$. Using $\gamma_B$ and $\phi$, we can now invoke Lemma \ref{condproblmm} to define a probability measure $\tilde{\gamma}$ on $\Omega_{M,N}^\circ \times \Omega_{M,N}^\circ$. We claim that $\tilde{\gamma}$ is a coupling of $\mu$ and $\mu'$. To see this, take any $S_1\in \mb(\Omega_B)$ and $S_2\in \mb(\Omega_B^*)$. Then 
\begin{align*}
\tilde{\gamma}(S_1\times \Omega_{B}\times S_2\times \Omega_B^* ) &= \int_{S_1\times \Omega_B} \phi((x,x'),S_2\times \Omega_B^*) d\gamma_B(x,x')\\
&= \int_{S_1\times \Omega_B} \mu_x(S_2) d\gamma_B(x,x')\\
&= \int_{S_1} \mu_x(S_2)d\mu_B(x) = \mu(S_1\times S_2),
\end{align*}
where the second-to-last identity holds because $\mu_B$ is a marginal probability of $\gamma_B$, as deduced above.  Similarly,
\begin{align*}
\tilde{\gamma}(\Omega_B\times S_1\times \Omega_B^*\times  S_2 ) &= \mu'(S_1\times S_2). 
\end{align*}
Thus, $\tilde{\gamma}$ is indeed a coupling of $\mu$ and $\mu'$. Let us denote $\tilde{\gamma}$ by $\tau_B(\gamma)$, viewing it as a function of $\gamma$. The map $\tau_B$ will be called the {\it local update map} corresponding to the cube $B$. Note that the definition of $\tau_B$ depends not only on $B$, but also on $M$, $N$, $\delta$, $\delta'$ and $r$, but we will consider those as fixed. 

For any coupling $\gamma$ of $\mu$ and $\mu'$, and any edge $e\in E_{M,N}^\circ$, let
\begin{align}\label{fdef}
\rho(\gamma, e) := \gamma(\{(\omega,\omega'): \omega_e\ne \omega_e'\}).
\end{align}
For any edge $e\in E_{M,N}^\circ$, let $U(e)$ be the set of all edges $u\in E_{M,N}^\circ$ such that $e$ and $u$ are both in some common cube of width $4r$, and let $V(e)$ be the set of all $u\in E_{M,N}^\circ$ such that both $e$ and $u$ are in some common cube of width $2N$. 
The following lemma estimates how $\rho$ changes under a local update map.
\begin{lmm}\label{localcoupling}
Let all notation be as above. Take any $B\in \fb$, any $e\in E_{M,N}^\circ$ and any coupling $\gamma$ of $\mu$ and $\mu'$. Let $E^\circ$ be the set of interior edges of $B$, and let $\partial'E$ be the spatial boundary of $B$. If $e\notin E^\circ$, then $\rho(\tau_B(\gamma),e)=\rho(\gamma,e)$. On the other hand, if $e\in E^\circ$, then 
\begin{align*}
\rho(\tau_B(\gamma), e) &\le C_1N^{d-1} e^{-C_2r}\sum_{u\in \partial' E\cap V(e)} \rho(\gamma,u)  + \sum_{u\in \partial'E \cap U(e) } \rho(\gamma, u) .
\end{align*}
\end{lmm}
\begin{proof}
For simplicity of notation, let $\tilde{\gamma} := \tau_B(\gamma)$. Take any $(\omega,\omega')\in \Omega_{M,N}^\circ \times \Omega_{M,N}^\circ$. Write $(\omega,\omega')$ as an element $(x,x', y, y')$ of $\Omega_B\times \Omega_B \times \Omega_B^*\times \Omega_B^*$. If $e\notin E^\circ$, then $\omega_e\ne \omega'_e$ means that $x_e\ne x_e'$. So in this case, by the definitions of $\gamma$ and $\tilde{\gamma}$, 
\begin{align*}
\rho(\tilde{\gamma}, e) &= \tilde{\gamma}(\{(\omega,\omega'):\omega_e\ne \omega_e'\}) \\
&= \gamma_B(\{(x,x'): x_e\ne x_e'\}) = \rho(\gamma, e).
\end{align*}
Next, suppose that $e\in E^\circ$. Then $\omega_e\ne \omega_e'$ means that $y_e\ne y_e'$. Thus,
\begin{align}
\rho(\tilde{\gamma}, e) &= \tilde{\gamma}(\{(\omega,\omega'):\omega_e\ne \omega_e'\})\notag \\
&= \int_{\Omega_B\times\Omega_B} \phi((x,x'), \{(y,y'): y_e\ne y_e'\}) d\gamma_B(x,x'). \label{badcond0}
\end{align}
Take any $(x,x')$. Let 
\[
p := \phi((x,x'), \{(y,y'): y_e\ne y_e'\}).
\]
Let $A$ be the set of $u\in \partial'E$ such that $x_u\ne x_u'$. Since $\delta$ and $\delta'$ agree on the temporal faces, and $B$ has no intersection with the spatial boundary of $S_{M,N}$, we see that $A$ consists of all boundary edges of $B$ where the boundary conditions defined by $x$ and $x'$ disagree. So by the construction of $\phi$ (using Lemma \ref{cubecoupling}), we see that if $U(e)$ has no intersection with $A$, then 
\[
p \le C_1 |A| r^{d-1} e^{-C_2r}\le C_1 N^{d-1} e^{-C_2r}|A|. 
\] 
If $U(e)$ intersects $A$, we simply use $p\le 1$.  Combining, we get that for any $(x,x')$ and $e$, 
\begin{align}\label{badp}
p &\le C_1 N^{d-1} e^{-C_2r}|A| + |U(e)\cap A|. 
\end{align}
But note that by the definition of $\gamma_B$,
\begin{align*}
 \int_{\Omega_B\times\Omega_B} |A| d\gamma_B(x,x') &= \int_{\Omega_{M,N}^\circ\times \Omega_{M,N}^\circ} |\{u\in \partial'E: \omega_u\ne \omega'_u\}| d\gamma(\omega, \omega')\\
 &= \sum_{u\in \partial'E} \gamma(\{(\omega,\omega'): \omega_u\ne \omega'_u\})\\
 &= \sum_{u\in \partial' E} \rho(\gamma, u). 
\end{align*}
We can replace $\partial'E$ by $\partial'E \cap V(e)$ in the above sum since the two sets are equal for any $e\in E^\circ$. Thus,
\begin{align}\label{badcond1}
\int_{\Omega_B\times\Omega_B} |A| d\gamma_B(x,x')  &= \sum_{u\in \partial' E\cap V(e)} \rho(\gamma, u).
\end{align}
Similarly,
\begin{align}\label{badcond2}
 \int_{\Omega_B\times\Omega_B}|U(e)\cap A|d\gamma_B(x,x')  &= \sum_{u\in\partial'E\cap U(e)} \rho(\gamma, u). 
\end{align}
Using the information obtained from \eqref{badp}, \eqref{badcond1} and \eqref{badcond2} in \eqref{badcond0}, we get the required upper bound.
\end{proof}

\section{Global update map on a slab}\label{globalsec}
Let us continue to use the notations introduced in the previous section. In particular, let us fix $M$, $N$, $\delta$, $\delta'$ and $r$ as before. We define the  {\it global update map} on the space of all couplings of $\mu$ and $\mu'$ as
\[
\tau(\gamma) := \frac{1}{|\fb|}\sum_{B\in \fb} \tau_B(\gamma). 
\]
If $\fb$ is empty (which happens if $M\le  N$), just let $\tau(\gamma):=\gamma$.  
For any edge $e$, let $\fb(e)$ be the set of all $B\in \fb$ such that $e$ is an interior edge of~$B$.  The following lemma quantifies how the global update map `improves' any given coupling. This is the third step towards the construction of the coupling in a slab. 
\begin{lmm}\label{globalupdate}
Suppose that $M> N$, so that $\fb$ is nonempty. Then for any coupling $\gamma$ of $\mu$ and $\mu'$, and any $e\in E_{M,N}^\circ$,
\begin{align*}
\rho(\tau(\gamma),e)&\le \biggl(1-\frac{|\fb(e)|}{|\fb|}\biggr)\rho(\gamma, e) +  \frac{C_1 N^{2d-3} e^{-C_2r}}{|\fb|} \sum_{u\in V(e)} \rho(\gamma,u)\\
&\qquad + \frac{C_3 N^{d-2}}{|\fb|}\sum_{u\in U(e)} \rho(\gamma, u). 
\end{align*} 
\end{lmm}
\begin{proof}
Note  that for any edge $e$, there can be at most $CN^{d-2}$ cubes $B\in \fb$ such that $e\in \partial' E$ (where, as before, $\partial' E$ denotes the spatial boundary of $B$). Therefore by Lemma \ref{localcoupling}, 
\begin{align*}
\sum_{B\in \fb} \rho(\tau_B(\gamma), e) &\le (|\fb| - |\fb(e)|) \rho(\gamma, e) + C_1 N^{2d-3} e^{-C_2r} \sum_{u\in V(e)} \rho(\gamma,u)\\
&\qquad + C_3 N^{d-2}\sum_{u\in U(e)} \rho(\gamma, u). 
\end{align*}
Since the function $\rho$ is linear in $\gamma$,
\begin{align*}
\rho(\tau(\gamma), e) &= \frac{1}{|\fb|}\sum_{B\in \fb} \rho(\tau_B(\gamma),e).
\end{align*}
Combining the above identity with the inequality from the previous display, we get the required bound. 
\end{proof}

\section{Coupling in a slab}\label{slabcouplingsec}
Let $M$, $N$, $\delta$, $\delta'$ and $r$ be fixed as in the previous two sections. We will now construct a coupling of $\mu$ and $\mu'$ using an infinite number of iterations of the global update map $\tau$. Start with the coupling $\gamma_0 := \mu\times \mu'$. For each $n$, let $\gamma_n := \tau^n(\gamma)$. 
Since $\{\gamma_n\}_{n\ge0}$ is a sequence of probability measures on the compact Polish space $\Omega_{M,N}^\circ \times \Omega_{M,N}^\circ$, it has a subsequence $\{\gamma_{n_k}\}_{k\ge 1}$ converging weakly to a limit $\gamma$.  It is not hard to see that  $\gamma$ is also a coupling of $\mu$ and $\mu'$. The following lemma gives the main property of this coupling that will be useful for us.
\begin{lmm}\label{slabcoupling}
There is some $N_0$ depending only on $G$, $\beta$ and $d$, such that if $N\ge N_0$, and $r$ is chosen to be $\lfloor (\log N)^2\rfloor$, then for any $M$ and any $e\in E_{M,N}^\circ$,
\begin{align}\label{fussy}
\rho(\gamma,e) &\le C_1 \exp\biggl(-\frac{C_2\dist(e,\partial' E_{M,N})}{N}\biggr),
\end{align}
where $\dist(e,\partial' E_{M,N})$ is the minimum value of $\dist(e,u)$ over all $u\in \partial' E_{M,N}$. 
\end{lmm}
\begin{proof}
Note that the distance of any edge $e$ to the spatial boundary of $S_{M,N}$ is bounded above by a constant times $M$, where the constant depends only on $d$. This shows that if $M\le N$, the proof is trivial, because the right side is bounded below by a positive constant depending only on $d$, whereas the left side is bounded above by $1$. So let us assume that $M> N$, which renders $\fb$ nonempty. For each $e$, let 
\[
p(e) := \limsup_{n\to\infty} \rho(\gamma_n, e).
\]
Since $\{(\omega,\omega'):\omega_e\ne \omega_e'\}$ is a relatively open subset of $\Omega_{M,N}^\circ \times \Omega_{M,N}^\circ$, the portmanteau lemma for weak convergence \cite[Theorem 3.9.1]{durrett10} gives us that 
\begin{align}\label{peineq}
\rho(\gamma, e) &\le \liminf_{k\to\infty} \rho(\gamma_{n_k},e)\notag \\
&\le \limsup_{k\to\infty} \rho(\gamma_{n_k},e) \le \limsup_{n\to\infty} \rho(\gamma_n,e)=p(e). 
\end{align} 
Thus, it suffices to get an upper bound for $p(e)$. Now, by Lemma \ref{globalupdate},
\begin{align*}
p(e)&\le \biggl(1-\frac{|\fb(e)|}{|\fb|}\biggr)p(e) +  \frac{C_1 N^{2d-3} e^{-C_2r}}{|\fb|} \sum_{u\in V(e)} p(u)\\
&\qquad + \frac{C_3 N^{d-2}}{|\fb|}\sum_{u\in U(e)} p(u)
\end{align*}
for every $e$. If $\fb(e)$ is nonempty, this can be rearranged as
\begin{align}\label{peinequality}
p(e)&\le  \frac{C_1 N^{2d-3} e^{-C_2r}}{|\fb(e)|} \sum_{u\in V(e)} p(u)+ \frac{C_3 N^{d-2}}{|\fb(e)|}\sum_{u\in U(e)} p(u).
\end{align}
Let $E'$ be the set of all $e\in E_{M,N}^\circ$ such that $\dist(e, \partial' E_{M,N})> 3N$. Clearly, it suffices to prove the required bound for $e\in E'$, because if $\dist(e, \partial' E_{M,N})\le  3N$, then the right side of \eqref{fussy} is bounded below by a positive constant depending only on $d$. So let us take any such $e$. It is easy to see that $|\fb(e)|> CN^{d-1}$.  Also, $|V(e)|\le CN^{d}$ and $|U(e)|\le Cr^{d}$. Therefore by \eqref{peinequality}, we get
\begin{align*}
p(e)&\le  C_1 N^{2d-2} e^{-C_2r} \max_{u\in V(e)} p(u)+ \frac{C_3 r^{d}}{N}\max_{u\in U(e)} p(u).
\end{align*}
Now, if $N$ is large enough (depending on $G$, $\beta$ and $d$) and $r = \lfloor (\log N)^2\rfloor$, then the coefficients in front both the maxima on the right are less than $1/4$. Since $U(e)\subseteq V(e)$, this implies that there is some $e_1\in V(e)$ such that
\[
p(e)\le \frac{p(e_1)}{2}.
\]
If $e_1$ is also in $E'$, then by the above argument applied to $e_1$ instead of $e$, we get an edge $e_2$ such that $p(e_1)\le p(e_2)/2$, and so on. Since $e_{i+1}\in V(e_i)$ for each $i$ (with $e_0=e$), it follows that each $e_{i+1}$ is within distance $CN$ from $e_i$. Therefore, the minimum $i$ such that $e_i\not\in E'$ must be at least $CN^{-1}\dist(e, \partial' E_{M,N})$. The final $p(e_i)$ can simply be bounded by $1$. This completes the proof of the lemma.
\end{proof}

%\section{An inequality for the Haar measure}\label{densitysec}

\section{Proof of Theorem \ref{main2}}
Let $N_0$ be as in Lemma \ref{slabcoupling}, and take any $N\ge N_0$ and any $M\ge 1$. Take any two boundary conditions $\delta$ and $\delta'$ for $S_{M,N}$ that agree on the temporal faces. Let $\mu$ and $\mu'$ be the probability measures defined by our lattice gauge theory on $S_{M,N}$ with these boundary conditions. 

By Lemma \ref{slabcoupling}, it is possible to construct a coupling of $\mu$ and $\mu'$ such that if $(\omega, \omega')$ is a pair of coupled configurations, then for any edge $e$ in the slab, the chance of $\omega_e\ne \omega_e'$ is exponentially small in the distance of $e$ from the spatial boundary. From this, it follows that if $f$ is a bounded measurable function of $\{\omega_e\}_{e\in A}$ for some fixed set $A$, then $|\int f d\mu - \int fd\mu'|$ falls off exponentially in the distance of $A$ from the spatial boundary of $S_{M,N}$. In particular, if we let $M\to \infty$, then it is impossible to produce a sequence of boundary conditions that differ only on the spatial boundaries, such that $\int fd\mu - \int fd\mu'$ does not tend to zero. And from this, it follows that for any boundary condition on the infinite slab $\{-N,\ldots,N\} \times \zz^{d-1}$, there is a unique Gibbs measure for our lattice gauge theory on the slab. Uniqueness of the Gibbs measure trivially implies that center symmetry cannot be spontaneously broken. This proves the first assertion of Theorem \ref{main2}. 

The above argument also implies that if $f$ is a vertical chain variable associated with the vertical chain through the center of $S_{M,N}$, its expected value in $S_{M,N}$ must tend to its expected value under the unique Gibbs measure on the infinite slab exponentially fast in $M$ as $M\to\infty$. But its expected value under the Gibbs measure must be zero due to center symmetry and the fact that $f$ transforms to $cf$ under the center transform defined in Section \ref{main1proof}. This shows that the function $V$ from Lemma~\ref{vertchainlmm} must have at least linear growth. Thus, from the proof of Theorem~\ref{main1}, we now see that in the setting of Theorem~\ref{main2}, we have
\begin{align}\label{wilson1}
|\smallavg{W_\ell}|\le e^{(C_1-C_2R)T}
\end{align}
for any rectangular loop with side-lengths $R\le T$. This almost proves the second assertion of  Theorem \ref{main2}, except that we have to remove the $C_1T$ term from the exponent. The first step towards eliminating this term is the following lemma. Recall that $\pi$ is a finite-dimensional irreducible unitary representation of $G$ which acts nontrivially on the center of $G$. 

%The first step towards eliminating the perimeter term in \eqref{wilson1} is to prove the following lemma about the pair $(G,\pi)$. It gives a lower bound on the fluctuations of any composition of $\pi$ followed by a  linear function under any  probability measure on $G$ whose density with respect to Haar measure is uniformly bounded above and below by positive constants. 
\begin{lmm}\label{haarlmm}
Let $\lambda_0$ denote the normalized Haar measure on $G$. Let $\rho$ be a probability density with respect to $\lambda_0$. Suppose that there are positive constants $a$ and $b$ such that $a\le \rho(g)\le b$ for all $g\in G$. Then there is some $\ve\in (0,1)$ depending only on $G$, $\pi$, $a$, and $b$, such that for any function $f:G\to\cc$ which is a composition of  $\pi$ followed by a linear map,
\[
\biggl|\int_G f(g)\rho(g)d\lambda_0(g)\biggr|^2 \le (1-\ve)^2 \int_G|f(g)|^2\rho(g) d\lambda_0(g).
\]
\end{lmm}
\begin{proof}
Let $\lambda_0^{\otimes 2}$ denote the normalized product Haar measure on $G\times G$. Then% note that
\begin{align}
&\frac{1}{2}\int_{G\times G} |f(g)-f(g')|^2 \rho(g)\rho(g')d\lambda_0^{\otimes 2}(g,g') \notag\\
&=\frac{1}{2}\int_{G\times G} (|f(g)|^2  - f(g)\overline{f(g')} - \overline{f(g)}f(g') + |f(g')|^2) \rho(g)\rho(g')d\lambda_0^{\otimes 2}(g,g') \notag \\
&= \int_G|f(g)|^2\rho(g) d\lambda_0(g) - \biggl|\int_G f(g)\rho(g)d\lambda_0(g)\biggr|^2.\label{trivialid}
\end{align}
Let $g_0$ be an element of the center of $G$ such that $\pi(g_0)$ is not the identity operator. As in the proof of Theorem \ref{main1}, we note that by Schur's lemma, $\pi(g_0)$ must be $cI$ for some $c\ne 1$. Since $f$ is the composition of $\pi$ followed by a linear map, the invariance of the Haar measure gives us 
\begin{align*}
&\int_{G\times G} |f(g)-f(g')|^2 \rho(g)\rho(g')d\lambda_0^{\otimes 2}(g,g')\\
&= \int_{G\times G} |f(g)-f(g_0g')|^2 \rho(g)\rho(g_0g')d\lambda_0^{\otimes 2}(g,g')\\
&= \int_{G\times G} |f(g)-cf(g')|^2 \rho(g)\rho(g_0g')d\lambda_0^{\otimes 2}(g,g').
\end{align*}
Since $a\le \rho(g)\le b$ for all $g$, we have
\[
\rho(g_0g') \ge a \ge \frac{a}{b}\rho(g'). 
\]
Applying this to the previous display gives 
\begin{align*}
&\int_{G\times G} |f(g)-f(g')|^2 \rho(g)\rho(g')d\lambda_0^{\otimes 2}(g,g')\\
&\ge \frac{a}{b} \int_{G\times G} |f(g)-cf(g')|^2 \rho(g)\rho(g')d\lambda_0^{\otimes 2}(g,g').
\end{align*}
But trivially, since $a/b\le 1$, 
\begin{align*}
&\int_{G\times G} |f(g)-f(g')|^2 \rho(g)\rho(g')d\lambda_0^{\otimes 2}(g,g')\\
&\ge \frac{a}{b} \int_{G\times G} |f(g)-f(g')|^2 \rho(g)\rho(g')d\lambda_0^{\otimes 2}(g,g').
\end{align*}
Combining the last two displays, we get
\begin{align*}
&\int_{G\times G} |f(g)-f(g')|^2 \rho(g)\rho(g')d\lambda_0^{\otimes 2}(g,g')\\
&\ge \frac{a}{2b} \int_{G\times G} (|f(g)-f(g')|^2+|f(g)-cf(g')|^2) \rho(g)\rho(g')d\lambda_0^{\otimes 2}(g,g').
\end{align*}
But by the inequality $|w-z|^2 \le 2|w|^2+2|z|^2$, we have
\[
|f(g)-f(g')|^2+|f(g)-cf(g')|^2 \ge \frac{|1-c|^2}{2}|f(g')|^2. 
\]
Therefore,
\begin{align*}
&\int_{G\times G} |f(g)-f(g')|^2 \rho(g)\rho(g')d\lambda_0^{\otimes 2}(g,g')\\
&\ge \frac{a|1-c|^2}{4b} \int_{G\times G} |f(g')|^2 \rho(g)\rho(g')d\lambda_0^{\otimes 2}(g,g')\\
&=\frac{a|1-c|^2}{4b} \int_{G} |f(g)|^2 \rho(g)d\lambda_0(g).
\end{align*}
The proof is now easily completed by combining the above inequality with the identity \eqref{trivialid}.
\end{proof}

%\section{Norm of conditional expectation}\label{condnormsec}
In what follows, we will need some basic facts about matrix norms. First, note that for any $A,B\in M_n(\cc)$, the Cauchy--Schwarz inequality implies that
\begin{align}\label{trab}
|\tr(AB)|\le \|A\|\|B\|.
\end{align}
Let $\|A\|_{op}$ denote the $L^2$ operator norm of a matrix $A\in M_n(\cc)$, defined as
\[
\|A\|_{op} := \sup\{\|Ax\|: x\in \cc^n, \|x\|=1\}. 
\]
It is easy to see that the operator norm satisfies
\begin{align}\label{opnormineq}
\|AB\|_{op}\le \|A\|_{op}\|B\|_{op}.
\end{align}
Take any $A,B\in M_n(\cc)$. Let $b_1,\ldots,b_n$ be the columns of $B$. Then
\[
\|AB\|^2 = \sum_{i=1}^n \|Ab_i\|^2 \le \sum_{i=1}^n\|A\|_{op}^2 \|b_i\|^2 = \|A\|_{op}^2 \|B\|^2.
\]
As a consequence of this inequality and the inequalities \eqref{trab} and \eqref{opnormineq}, we get that for any sequence $A_1,\ldots,A_k\in M_n(\cc)$, where $k\ge 3$, 
\begin{align}
|\tr(A_1A_2\cdots A_k)| &\le \|A_1\cdots A_{k-1}\|\|A_k\| \notag \\
&\le \|A_1\|_{op}\|A_2\|_{op}\cdots \|A_{k-2}\|_{op}\|A_{k-1}\|\|A_k\|.\label{normineq}
\end{align}
We will now use Lemma \ref{haarlmm} to show that for any edge $e$, the conditional expectation of $\pi(\omega_e)$ given $\{\omega_u: u\ne e\}$ is a matrix whose operator norm  is strictly less than $1$. Moreover, the gap is uniformly bounded below by a constant. %This is the key ingredient in the proof of the perimeter law, which will be presented in the next section.
\begin{lmm}\label{conditionallmm}
Consider any Gibbs measure for our lattice gauge theory on $\zz^d$. For any edge $e$, let $\smallavg{\pi(\omega_e)}'$ denote the matrix of conditional expectations of the entries of $\pi(\omega_e)$ given $\{\omega_u:u\ne e\}$. There is a constant $\ve\in (0,1)$, depending only on $G$, $\beta$, $\pi$ and $d$, such that $\|\smallavg{\pi(\omega_e)}'\|_{op}\le 1-\ve$. 
\end{lmm}
\begin{proof}
It is not hard to see that the conditional probability density  (with respect to Haar measure)  of $\omega_e$ given $\{\omega_u:u\ne e\}$ is bounded above and below by two positive constants $a$ and $b$, which depend only on $G$, $\beta$ and $d$. Take any $x\in \cc^m$ such that $\|x\|=1$. Then each component of the vector $\pi(\omega_e) x$ is a linear function of $\pi(\omega_e)$. Therefore by Lemma \ref{haarlmm},
\begin{align*}
\|\smallavg{\pi(\omega_e) x}'\|^2 &\le (1-\ve)^2 \smallavg{\|\pi(\omega_e)x\|^2}', 
\end{align*}
where $\ve \in (0,1)$ depends only on $G$, $\beta$, $\pi$ and $d$. Since $\pi(\omega_e)$ is a unitary matrix, $\|\pi(\omega_e) x\|=\|x\|=1$. Thus, $\|\smallavg{\pi(\omega_e) }'x\| = \|\smallavg{\pi(\omega_e) x}'\|\le 1-\ve$. Taking supremum over $x$, we get the desired result.
\end{proof}

%\section{Perimeter law}\label{perimetersec}
We will now use the previous lemma to prove the perimeter law for Wilson loop expectations.  A proof of the perimeter law already appears in an old paper of \citet{simonyaffe82}; however, that result is conditional on a certain uniqueness assumption about the Gibbs measure. The result given below is unconditional, showing that the perimeter law upper bound holds in complete generality.
\begin{lmm}\label{perilmm}
For any rectangular loop $\ell$ with side-lengths $R$ and $T$, $|\smallavg{W_\ell}|\le C_1 e^{-C_2(R+T)}$ for any Gibbs measure of our lattice gauge theory on $\zz^d$.
\end{lmm}
\begin{proof}
Without loss of generality, suppose that $R\le T$. Let $e_1,e_2,\ldots,e_k$ be the edges of $\ell$, where the first $T$ edges belong to a side of length $T$. Let $\smallavg{\cdot}'$ denote the conditional expectation given $\{\omega_e:e\notin \{e_1,\ldots, e_T\}\}$. Under this conditioning, $\omega_{e_1},\ldots,\omega_{e_T}$ are independent random matrices because no two of these edges share a common plaquette. Moreover, for any $1\le i\le T$, the conditional distribution of $\omega_{e_i}$ given $\{\omega_e:e\notin\{e_1,\ldots,e_T\}\}$ is the same as the conditional distribution of $\omega_{e_i}$ given $\{\omega_e:e\ne e_i\}$. Applying the conditional independence to each summand in the formula for the trace in the following display, we get
\begin{align*}
\smallavg{W_\ell}' &= \smallavg{\tr(\pi(\omega_{e_1})\cdots\pi(\omega_{e_k}))}'\\
&= \tr(\smallavg{\pi(\omega_{e_1})}'\smallavg{\pi(\omega_{e_2})}'\cdots \smallavg{\pi(\omega_{e_T})}' \pi(\omega_{e_{T+1}})\cdots \pi(\omega_{e_k})). 
\end{align*}
Thus, by Lemma \ref{conditionallmm}, the matrix inequality \eqref{normineq} (observing that $k-3\ge r$), and the fact that $\|\pi(\omega_e)\|_{op}=1$ for all $e$ (since $\pi$ is a unitary representation), we get
\begin{align*}
|\smallavg{W_\ell}'| &= |\tr(\smallavg{\pi(\omega_{e_1})}'\smallavg{\pi(\omega_{e_2})}'\cdots \smallavg{\pi(\omega_{e_T})}' \pi(\omega_{e_{T+1}})\cdots \pi(\omega_{e_k}))|\\
&\le  \|\smallavg{\pi(\omega_{e_1})}'\|_{op}\cdots \|\smallavg{\pi(\omega_{e_T})}'\|_{op} \|\pi(\omega_{e_{T+1}})\|_{op}\\
&\qquad \qquad \qquad \cdots \|\pi(\omega_{e_{k-2}})\|_{op} \|\pi(\omega_{e_{k-1}})\|\|\pi(\omega_{e_k})\|\\
&\le C(1-\ve)^T\le C(1-\ve)^{(R+T)/2}, 
\end{align*}
where $C>0$ and $\ve\in (0,1)$ depend only on $G$, $\beta$, $\pi$ and $d$. Thus, the same bound holds for $|\smallavg{W_\ell}|$. 
\end{proof}

%\section{Completing the proof of Theorem \ref{mainthm}}

Finally, we are ready to complete the proof of the second assertion of Theorem~\ref{main2}. Let $\ell$ be a rectangular loop with side-lengths $R\le T$. Consider any Gibbs measure for our lattice gauge theory on $\zz^d$. By inequality \eqref{wilson1}, 
\begin{align}\label{last1}
|\smallavg{W_\ell}|\le e^{C_1T - C_2 RT}.
\end{align}
On the other hand, by Lemma \ref{perilmm},
\begin{align}\label{last2}
|\smallavg{W_\ell}| \le C_3 e^{-C_4(R+T)}\le C_3 e^{-C_4 T}. 
\end{align}
Let $a:= C_4/(C_1+C_4)$. Since $a\in [0,1]$, we may combine \eqref{last1} and \eqref{last2} as follows:
\begin{align*}
|\smallavg{W_\ell}| &\le ( e^{C_1T - C_2 RT})^a (C_3 e^{-C_4 T})^{1-a}\\
&= C_3^{1-a} e^{-C_2a RT},
\end{align*}
where the last equality holds because $C_1 a - C_4(1-a)=0$. This completes the proof of Theorem \ref{main2}.

\section*{Acknowledgments}
I thank  Christian Borgs, Persi Diaconis, J\"urg Fr\"ohlich, Len Gross, Erhard Seiler,  Senya Shlosman, Tom Spencer, Raghu Varadhan, and Akshay Venkatesh for helpful discussions. I am especially grateful to Edward Witten and Steve Shenker for many lengthy and illuminating conversations, and to Sky Cao for carefully reading the proof and pointing out some important references. Lastly, I thank the referees for a number of useful suggestions.%, and to Sky Cao for carefully checking the proof and pointing out the reference~\cite{borgs84}. %, which helped resolve an important issue. 

\end{document}